\documentclass[11pt,reqno]{amsart}
\usepackage{latexsym,amssymb,graphicx,multicol,mathrsfs,color,xypic,amscd, amsxtra, verbatim}

\theoremstyle{plain}
\newtheorem*{claim}{Claim}

\newtheorem*{Step1}{Step 1}
\newtheorem*{Step2}{Step 2}
\newtheorem*{Step3}{Step 3}

\newtheorem{theorem}{Theorem}[section]
\newtheorem*{thm}{Theorem}
\newtheorem{proposition}[theorem]{Proposition}
\newtheorem{lemma}[theorem]{Lemma}

\newtheorem{corollary}[theorem]{Corollary}

\theoremstyle{definition}
\newtheorem*{remark}{Remark}
\newtheorem*{remarks}{Remarks}
\newtheorem*{problem}{Problem}
\newtheorem*{mainresult}{Main Result}
\newtheorem{definition}[theorem]{Definition}
\newtheorem*{defn}{Definition}

\theoremstyle{remark}
\newtheorem*{Open Question}{Open Question}
\newtheorem{example}{Example}

\input xypic
\xyoption{all}

\begin{document}
\def\A{\mathcal{A}}
\def\Aut{\text{Aut\,}}
\def\Ann{\text{Ann\,}}
\def\Bl{\text{\rm Bl}}
\def\codim{\text{\rm codim}}
\def\Gal{\text{Gal\,}}
\def\tphi{\tilde{\phi}}
\def\dim{\text{dim\,}}
\def\characteristic{\text{characteristic\,}}
\def\discrep{\text{discrep}}
\def\Def{\text{Def\,}}
\def\LDef{\overline{\text{Def}}\,}
\def\Tbar{\underline{T}}
\def\length{\text{\rm length}}
\def\B{\mathcal{B}}
\def\C{\mathcal{C}}
\def\D{\text{Def\,}}
\def\div{\text{\rm div}}
\def\e{\epsilon}
\def\ev{\text{\rm ev}}
\def\E{\mathcal{E}}
\def\Eq{\text{Eq\,}}
\def\Exc{\text{Exc\,}}
\def\Eff{\text{Eff\,}}
\def\Ext{\mathscr{E}\!xt}
\def\F{\mathscr{F}}
\def\Frac{\text{Frac\,}}
\def\G{\mathscr{G}}
\def\H{\mathscr{H}}
\def\Hom{\mathscr{H}\!om}
\def\hom{\text{\rm Hom}}
\def\I{\mathscr{I}}
\def\Ind{\text{Indet\,}}
\def\Image{\text{\rm Image}}
\def\Isom{\text{\rm Isom}}
\def\id{\text{\rm id}}
\def\im{\text{im\,}}
\def\J{\mathscr{J}}
\def\K{\mathscr{K}}
\def\Ker{\text{\rm Ker}}
\def\L{\mathscr{L}}
\def\M{\overline{M}}
\def\Moduli{\text{\underline{Attaching Moduli}}}
\def\Maps{\text{\underline{Attaching Maps}}}
\newcommand{\Mod}[2][n]{\ensuremath{\overline{\mathcal{M}}_{1,#1}(#2)}}
\newcommand{\m}{\ensuremath{m}}
\def\SM{\overline{\mathcal{M}}}
\def\NSM{\widetilde{\mathcal{M}}}
\def\NM{\widetilde{M}}
\def\NE{\overline{\text{NE}}}
\def\Nef{\text{Nef}}
\def\O{\mathscr{O}}
\def\R{\mathcal{R}}
\def\T{\mathcal{T}}
\def\ttau{\tilde{\tau}}
\def\P{\mathbb{P}}
\def\Q{\mathbb{Q}}
\def\X{\mathcal{X}}
\def\U{\mathcal{U}}
\def\V{\mathcal{V}}
\def\Z{\mathcal{Z}}
\def\pic{\text{\rm Pic}\,}
\def\S{\mathcal{S}}
\def\Sch{\text{Sch}}
\def\SA{(\mathcal{S}, \mathcal{A})}
\def\red{_{\text{\rm red}}}
\def\Res{\text{\rm Res}}
\def\Spec{\text{\rm Spec\,}}
\def\Proj{\text{\rm Proj\,}}
\def\Supp{\text{\rm Supp\,}}
\def\Mor{\text{\rm Mor}}
\def\Pic{\text{\rm Pic\,}}
\def\Ver{\text{\rm Ver}}

\title[Modular compactifications of $\mathcal{M}_{1,n}$]{Modular compactifications of $\mathcal{M}_{1,n}$ I: \\Construction of $\SM_{1,\A}(m)$}
\author{David Ishii Smyth}

\maketitle
\begin{abstract}
We introduce a sequence of isolated curve singularities, the elliptic $m$-fold points, and an associated sequence of stability conditions, generalizing the usual definition of Deligne-Mumford stability. For every pair of integers $1 \leq m < n$, we prove that the moduli problem of $n$-pointed $m$-stable curves of arithmetic genus one is representable by a proper irreducible Deligne-Mumford stack $\SM_{1,n}(m)$. We also consider weighted variants of these stability conditions, and construct the corresponding moduli stacks $\SM_{1,\A}(m)$. In forthcoming work, we will prove that these stacks have projective coarse moduli and use the resulting spaces to give a complete description of the log minimal model program for $\M_{1,n}$.\\
\end{abstract}

\tableofcontents
\pagebreak

\section{Introduction}
\subsection{Why genus one?}
One of the most beautiful and influential theorems of modern algebraic geometry is

\begin{thm}[Deligne-Mumford \cite{DM}] The moduli stack of stable curves of arithmetic genus $g \geq 2$ is a smooth proper Deligne-Mumford stack over $\Spec(\mathbb{Z})$.
\end{thm}

The essential geometric content of the theorem is the identification of a suitable class of singular curves, namely \emph{Deligne-Mumford stable curves}, with the property that every incomplete one-parameter family of smooth curves has a unique `limit' contained in this class. The definition of a Deligne-Mumford stable curve comprises one local condition and one global condition.

\begin{defn}[Stable curve] A connected, reduced, complete curve $C$ is \emph{stable} if
\begin{itemize}
\item[(1)]$C$ has only nodes as singularities. (Local Condition) 
\item[(2)]$C$ satisfies the following two equivalent conditions. (Global Condition) 
\begin{itemize}
\item[(a)] $H^0(C, \Omega_{C}^{\vee})=0$.
\item[(b)] $\omega_{C}$ is ample.
\end{itemize}
\end{itemize}
\end{defn}

While the class of stable curves gives a natural modular compactification of the space of smooth curves, it is not unique in this respect. Using geometric invariant theory, Schubert constructed a proper moduli space for \emph{pseudostable curves} \cite{S91}.

\begin{defn}[Pseudostable curve] A connected, reduced, complete curve $C$ is \emph{pseudostable} if
\begin{itemize}
\item[(1)]$C$ has only nodes and cusps as singularities. (Local Condition)
\item[(2)] If $E \subset C$ is any connected subcurve of arithmetic genus one,\\ then $|E \cap \overline{C \backslash E}|  \geq 2.$ (Global Condition)
\item[(3)]$C$ satisfies the following two equivalent conditions. (Global Condition) 
\begin{itemize}
\item[(a)] $H^0(C, \Omega_{C}^{\vee})=0$.
\item[(b)] $\omega_{C}$ is ample.
\end{itemize}
\end{itemize}
\end{defn}
Notice that the definition of pseudostability involves a trade-off: the local condition has been weakened to allows cusps, while the global condition has been strengthened to disallow elliptic tails. It is easy to see how this trade-off comes about: As one ranges over all one-parameter smoothings of a cuspidal curve $C$, the associated limits in $\M_{g}$ are precisely curves of the form $\tilde{C} \cup E$, where $\tilde{C}$ is the normalization of $C$ and $E$ is an elliptic curve (of arbitrary $j$-invariant) attached to $\tilde{C}$ at the point lying above the cusp. Thus, any separated moduli problem must exclude either cusps or elliptic tails. In light of Schubert's construction, it is natural to ask

\begin{problem}
Given a reasonable local condition, e.g. a deformation-open collection of isolated curve singularities, is there a corresponding global condition which yields a proper moduli space?
\end{problem}

Any investigation of the above problem should begin by asking: which are the simplest isolated curve singularities? Let $C$ be a reduced curve over an algebraically closed field $k$, $p \in C$ a singular point, and $\pi:\tilde{C} \rightarrow C$ is the normalization of $C$ at $p$. We have two basic numerical invariants of the singularity:
\begin{align*}
\delta&=\dim_{k} (\pi_*\O_{\tilde{C}}/\O_{C}),\\
m&=|\pi^{-1}(p)|.
\end{align*}
$\delta$ may be interpreted as the number of conditions for a function to descend from $\tilde{C}$ to $C$, while $m$ is the number of branches. Of course, if a singularity has $m$ branches, there are $m-1$ obviously necessary conditions for a function $f \in \O_{\tilde{C}}$ to descend: $f$ must have the same value at each point in $\pi^{-1}(p)$. Thus, $\delta-m+1$ is the number of conditions for a function to descend \emph{beyond the obvious ones}, and we take this as the most basic numerical invariant of a singularity.
\begin{defn}
The \emph{genus} of an isolated singularity is $g:=\delta-m+1$.
\end{defn}
We use the name `genus' for the following reason: If $\C \rightarrow \Delta$ is a one-parameter smoothing of an isolated curve singularity $p \in C$, then (after a finite base-change) one may apply stable reduction around $p$ to obtain a proper birational morphism
\[
\xymatrix{
\C^{s} \ar[rr]^{\phi} \ar[dr]&&\C \ar[dl]\\
&\Delta&
}
\]
where $\C^{s} \rightarrow \Delta$ is a nodal curve, and $\phi(\Exc(\phi))=p$. Then it is easy to see that the genus of the isolated curve singularity $p \in C$ is precisely the arithmetic genus of the curve $\phi^{-1}(p)$. Thus, just as elliptic tails are replaced by cusps in Schubert's moduli space of pseudostable curves, any separated moduli problem allowing singularities of genus $g$ must disallow certain subcurves of genus $g$.

The simplest isolated curve singularities are those of genus zero. For each integer $m \geq 2$, there is a unique singularity with $m$ branches and genus zero, namely the union of the $m$ coordinate axes in $\mathbb{A}^{m}$. For our purposes, however, these singularities have one very unappealing feature: for $m \geq 3$, they are not \emph{Gorenstein}. This means that the dualizing sheaf of a curve containing such singularities is not invertible. Thus, a moduli problem involving these singularities has no obvious canonical polarization. For this reason, we choose to focus upon the next simplest singularities, namely those of genus one. It turns out that, for each integer $m \geq 1$, there is a unique Gorenstein curve singularity with $m$ branches and genus one (Proposition \ref{P:genus1}). These are defined below and pictured in Figure \ref{singularitiespic}.
\begin{figure}
\scalebox{.50}{\includegraphics{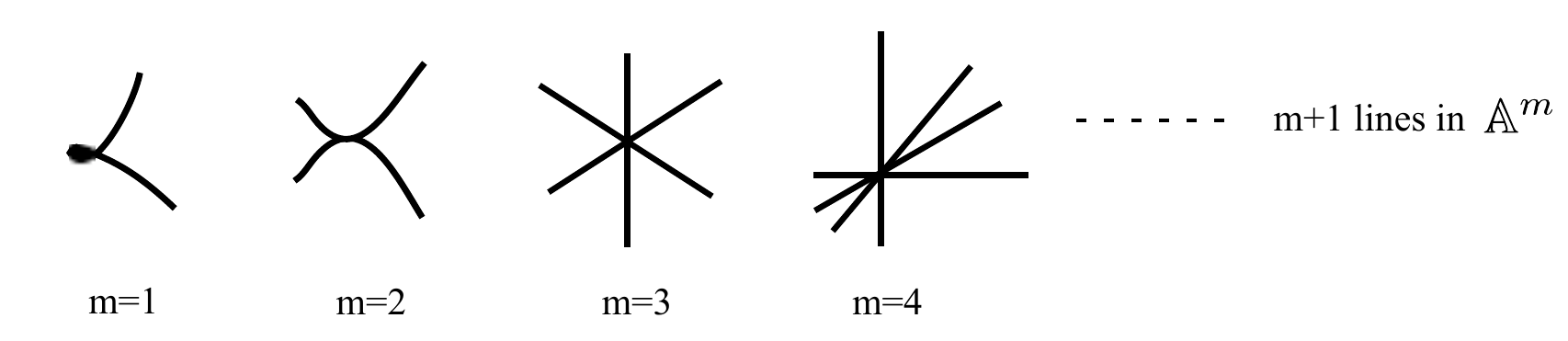}}
\caption{The sequence of elliptic $m$-fold points, the unique Gorenstein singularities of genus one.}\label{singularitiespic}
\end{figure}
\begin{defn}[The elliptic $m$-fold point] We say that $p$ is an \emph{elliptic $m$-fold point} of $C$ if
\begin{align*}
\hat{O}_{C,p} \simeq &
\begin{cases}
k[[x,y]]/(y^2-x^3) & m=1\,\,\,\text{(ordinary cusp)}\\
k[[x,y]]/(y^2-yx^2) & m=2 \,\,\,\text{(ordinary tacnode)} \\
k[[x,y]]/(x^2y-xy^2) & m=3 \,\,\, \text{(planar triple-point)}\\
k[[x_1, \ldots, x_{m-1}]]/I_m & m \geq 4, \text{($m$ general lines through the origin in $\mathbb{A}^{m-1}$),}\\
\end{cases}\\
&\,\,\,\,\, I_{m}:= \left(  x_{h}x_i-x_{h}x_j \, : \,\, i,j,h \in \{1, \ldots, m-1\} \text{ distinct} \right).
\end{align*}
\end{defn}

We will show that if $C$ is a curve with a single elliptic $m$-fold point $p$, then, as one ranges over all one-parameter smoothings of $C$, the associated limits in $\M_{g}$ are curves of the form $\tilde{C} \cup E$, where $\tilde{C}$ is the normalization of $C$ and $E$ is any stable curve of arithmetic genus one attached to $\tilde{C}$ at the points lying above $p$. Following Schubert, one is now tempted to define a sequence of moduli problems in which certain arithmetic genus one subcurves are replaced by elliptic $m$-fold points.

\begin{figure}
\scalebox{.50}{\includegraphics{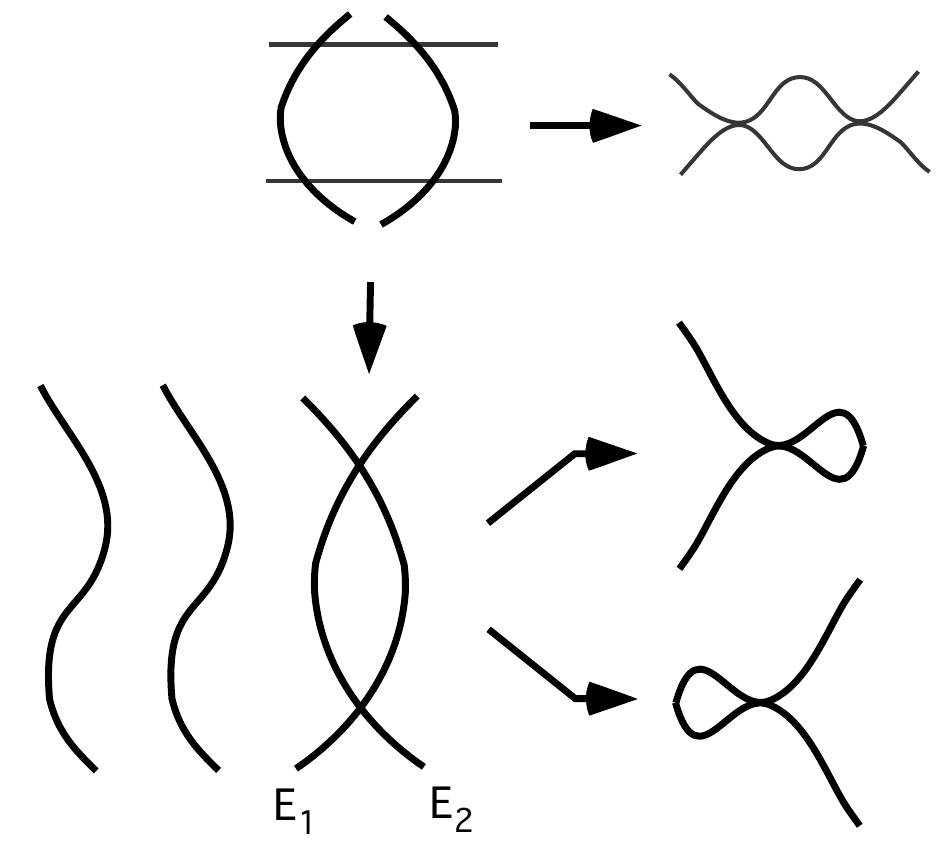}}
\caption{Three candidates for the `2-stable limit' of a one-parameter family of genus three curves specializing to a pair of elliptic bridges.}\label{F:manylimits}
\end{figure}

The idea seems plausible until one encounters the example pictured in figure \ref{F:manylimits}. There we see a one-parameter family of smooth genus three curves specializing to a pair of elliptic bridges, and we consider the question: How can one modify the special fiber to obtain a `tacnodal limit' for this family? Assuming the total space of the family is smooth, one can contract either $E_1$ or $E_2$ to obtain two non-isomorphic tacnodal special fibers, but there is no canonical way to distinguish between these two limits. A third possibility is to blow-up the two points of intersection $E_1 \cap E_2$, make a base-change to reduce the multiplicities of the exceptional divisors, and then contract \emph{both} elliptic bridges to obtain a bi-tacnodal limit whose normalization comprises a pair of smooth rational curves. This limit curve certainly appears canonical, but it has an infinite automorphism group and contains the other two pictured limits as deformations. This example suggests that a systematic handling of mildly non-separated moduli functors, either via the formalism of geometric invariant theory or Artin stacks, will be necessary in order to proceed at this level of generality. (See \cite{H:tacnodes} for a geometric invariant theory construction involving tacnodal curves.)

Happily, there is one non-trivial case in which this difficulty of multiple interacting elliptic components does not appear, namely the case of $n$-pointed stable curves of arithmetic genus one. This leads us to make the definition

\begin{defn}[$m$-stability]
Fix positive integers $m < n$. Let $C$ be a connected, reduced, complete curve of arithmetic genus one, let $p_1, \ldots, p_n$ be $n$ distinct smooth marked points, and let $\Sigma \subset C$ be the divisor $\sum_{i}p_i$. We say that $(C,p_1, \ldots, p_n)$ is \emph{$m$-stable} if
\begin{itemize}
\item[(1)] $C$ has only nodes and elliptic $l$-fold points, $l \leq m$, as singularities.
\item[(2)] If $E \subset C$ is any connected subcurve of arithmetic genus one, then $$|E \cap \overline{C \backslash E}|+|E \cap \Sigma|  >m.$$
\item[(3)] $H^0(C, \Omega_{C}^{\vee}(-\Sigma))=0.$
\end{itemize}
\end{defn}

\begin{figure}
\scalebox{.70}{\includegraphics{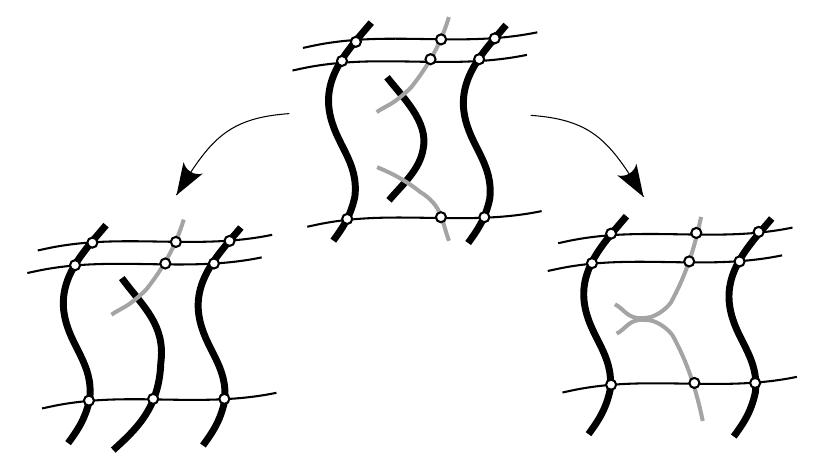}}
\caption{The 2-stable limit of a family of smooth curves acquiring an elliptic tail with marked point.}\label{F:simplelimit}
\end{figure}

\begin{figure}
\scalebox{.70}{\includegraphics{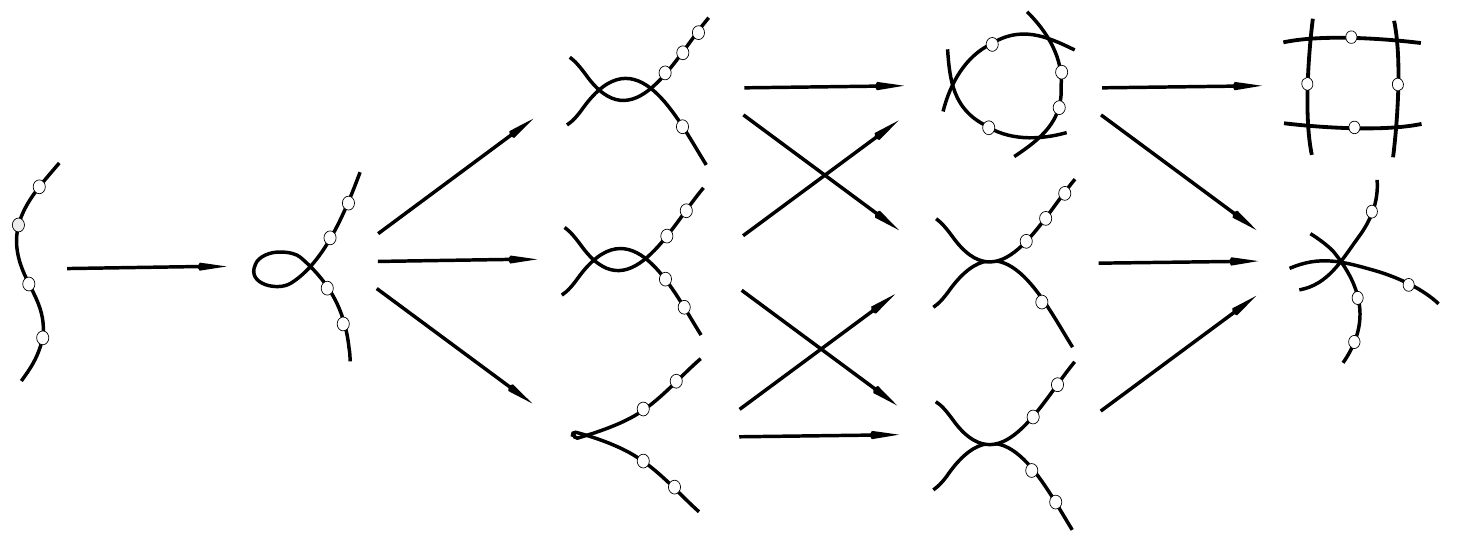}}
\caption{Equisingular stratification of $\M_{1,4}(3)$.}\label{F:BoundaryM14}
\end{figure}
\begin{remarks}
\begin{itemize}
\item[]
\item[(1)] The reason for considering $|E \cap \overline{C \backslash E}|+|E \cap \Sigma|$ rather than simply $|E \cap \overline{C \backslash E}|$ stems from the necessity of keeping marked points distinct. If, for example, one wishes to allow tacnodes into the moduli problem, one must disallow not only elliptic bridges, but also elliptic tails containing a single marked point. To obtain the 2-stable limit of the family of smooth curves whose Deligne-Mumford stable limit contains an elliptic tail with marked point, we blow-up the marked point lying on the elliptic component, and then contract the strict transform of this component to obtain a tacnodal special fiber. This process is pictured in Figure \ref{F:simplelimit}, where irreducible components of arithmetic genus one are pictured in black, while components of arithmetic genus zero are pictured in grey.
\item[(2)] The condition that $(C,p_1, \ldots, p_n)$ have no infinitesimal automorphisms is \emph{not} simply that every smooth rational component have three distinguished points. Furthermore, while
$$H^0(C, \Omega_{C}^{\vee}(-\Sigma))=0 \implies \omega_{C}(\Sigma) \text{ is ample,}$$ these conditions are not equivalent. These issues are addressed in Definition \ref{D:mstable}, where we reformulate the condition $H^0(C, \Omega_{C}^{\vee}(-\Sigma))=0$ in terms of distinguished points on each irreducible component of $\tilde{C}$.
\item[(3)] In order to provide some visual intuition for the definition of $m$-stability, we have supplied a diagram of the topological types (more precisely, the equisingularity classes) of $4$-pointed $3$-stable curves. In Figure \ref{F:BoundaryM14}, every irreducible component of every singular curve is rational, and the arrows indicate specialization relations between the various topological types. We leave it as an exercise to the reader (one which should be much easier after reading this paper) to verify that the pictured curves comprise \emph{all} topological types of 4-pointed 3-stable curves.
 \end{itemize}
\end{remarks}

The definition of $m$-stability is compatible with the definition of $\A$-stability introduced by Hassett \cite{Hassett4}, in which sections of low weight are allowed to collide. More precisely, we have

\begin{defn}[$(m,\A)$-stability]
Fix positive integers $m < n$, and an $n$-tuple of rational weights $\A=(a_1, \ldots, a_n) \in (0,1]^n$. Let $C$ be a connected, reduced, complete curve of arithmetic genus one, let $p_1, \ldots, p_n \in C$ be smooth (not necessarily distinct) points of $C$, and let $\Sigma$ denote the support of the $\Q$-divisor $\sum_{i}a_ip_i$. We say that $(C,p_1, \ldots, p_n)$ is \emph{$(m, \A)$-stable} if
\begin{itemize}
\item[(1)] $C$ has only nodes and elliptic $l$-fold points, $l\leq m$, as singularities.
\item[(2)] If $E \subset C$ is any connected subcurve of arithmetic genus one, then $$|E \cap \overline{C \backslash E}|+|E \cap \Sigma|  >m.$$
\item[(3)] $H^0(C, \Omega_{C}^{\vee}(-\Sigma))=0.$
\item[(4)] If $p_{i_1}=\ldots=p_{i_k} \in C$ coincide, then $\sum_{j=1}^{k}a_{i_j} \leq 1.$ 
\item[(5)] $\omega_{C}(\Sigma_ia_ip_i)$ is an ample $\Q$-divisor.
\end{itemize}
\end{defn}

The definition of an $(m,\A)$-stable curve extends to a moduli functor in the usual way, and we obtain

\begin{mainresult}
$\SM_{1,\A}(m)$, the moduli stack of $(m, \A)$-stable curves, is a proper irreducible Deligne-Mumford stack over $\Spec \mathbb{Z}[1/6]$.
\end{mainresult}
\begin{remark}
The restriction to $\Spec \mathbb{Z}[1/6]$ is due to the existence of `extra' infintesimal automorphisms of cuspidal curves in characteristics two and three, a phenomenon which is addressed in section \ref{S:tangentsheaf}.
\end{remark}

A major impetus for studying alternate compactifications of moduli spaces of curves comes from the program introduced by Brendan Hassett \cite{H2}, where one seeks modular descriptions for certain log-canonical models of $\M_{g,n}$. While special cases of this program have been worked out using geometric invariant theory \cite{H05,H:tacnodes,HL}, our construction gives the first example of an infinite sequence of singularities giving rise to alternate stability conditions. Our methods are also quite different from \cite{H05, H:tacnodes, HL} in that, rather than rely on geometric invariant theory to dictate our choice of moduli problem, we undertake a sufficiently in-depth investigation of the elliptic $m$-fold point to make moduli problems involving these singularities accessible via standard stack-theoretic techniques.  Our long-term goal is a systematic classification of alternate compactifications of $M_{g,n}
$ for all $g$.

In forthcoming work, we will study $\SM_{1,\A}(m)$ in the framework of birational geometry  \cite{me2}. In particular, we will develop techniques for doing intersection theory on $\SM_{1,\A}(m)$, construct explicit ample divisors on the associated coarse moduli spaces $\M_{1,\A}(m)$, and show that the rational maps $\M_{1,\A}(m) \dashrightarrow \M_{1,\A}(m+1)$ are Mori flips. This will enable us to give a complete description of the log minimal model program for $\M_{1,n}$, i.e. for all $\alpha \in \Q \cap [0,1]$ such that $K_{\SM_{1,n}}+\alpha\Delta$ is big, we will show that
$$\M_{1,n}(\alpha):=\Proj \oplus_{m>>0} H^0(\SM_{1,n}, m(K_{\SM_{1,n}} +\alpha\Delta))$$
is the normalization of the coarse moduli space of one of the moduli problems $\SM_{1,\A}(m)$ introduced in this paper. 

\subsection{Outline of results} In section 2, we investigate local properties of the elliptic $m$-fold point which are necessary for the construction of moduli. In section \ref{S:tangentsheaf}, we show that sections of $\Omega_{C}^{\vee}$  around an elliptic $m$-fold point $p \in C$ are given by regular vector fields on the normalization which vanish \emph{and} have identical first derivatives at the points lying above $p$. This will allow us to translation the condition $H^0(C, \Omega_{C}^{\vee}(-\Sigma))=0$ into a concrete statement involving the number of distinguished points on each irreducible component of $C$. In section \ref{S:dualizingsheaf}, we show that $\omega_{C}$ is invertible around an elliptic $m$-fold point $p$, and is generated by a rational differential on $\tilde{C}$ with double poles along the points lying above $p$. This implies that $\omega_{C}$ (twisted by marked points) is ample on any $m$-stable curve so that our moduli problem is canonically polarized. In section \ref{S:stablelimits}, we classify the collection of all `semistable tails' (Definition \ref{D:tail}) obtained by performing semistable reduction on the elliptic $m$-fold point (note that our definition of semistable reduction stipulates that the total space should be smooth). This set can be considered as an invariant associated to any smoothable isolated curve singularity. While the aforementioned fact that all $m$-pointed stable curves of genus one arise as stable limits of the elliptic $m$-fold point is an easy consequence of our analysis, we emphasize that we are classifying semistable limits, not merely stable limits, and this keeps track of extra information. For instance, the indices of the $A_{n}$-singularities appearing on the total space of the stable reduction are tracked by the length of the semistable chains appearing in the semistable reduction. These semistable limits turn out to satisfy a very delicate property: they are \emph{balanced} (Proposition \ref{P:semistablelimits}). This will be the key point in verifying that the moduli space of $m$-stable curves is separated.

In section 3, we construct $\SM_{1,\A}(m)$, the moduli space of $(m, \A)$-stable curves as a Deligne-Mumford stack over $\Spec \mathbb{Z}[1/6]$. In section \ref{S:fundecomp}, we prove some elementary topological facts about a reduced connected Gorenstein curve $C$ of arithmetic genus one. The key point is that $C$ admits a decomposition
$$C=Z \cup R_{1} \cup \ldots \cup R_{k},$$
where $Z$ is the unique connected arithmetic genus one subcurve of $C$ with no disconnecting nodes, and $R_{1}, \ldots, R_{k}$ are connected nodal curves of arithmetic genus zero (i.e. trees of $\P^{1}$'s). Furthermore, if $C$ possesses an elliptic $l$-fold point $p$, then $p$ is the unique non-nodal singularity of $C$, and $Z$ consists of $l$ smooth rational curves meeting at $p$. We call $Z$ the \emph{minimal elliptic subcurve of $C$} and its uniqueness is the essential reason that we can formulate a good moduli problem for genus one curves, but not in higher genus. In section \ref{S:mstable}, we define the moduli problem of $(m,\A)$-stable curves and prove that it is bounded and deformation-open. Following standard arguments, we obtain a moduli stack $\SM_{1,\A}(m)$. 

In sections \ref{S:ValuativeCriterion} and \ref{S:mAstability}, we verify the valuative criterion for $\SM_{1,\A}(m)$. In section \ref{S:ValuativeCriterion}, we verify the valuative criterion for $\SM_{1,n}(m)$, i.e. in the special case where $\A=(1, \ldots, 1)$. To show that one-parameter families of smooth curves possess an $m$-stable limit, we start from a Deligne-Mumford semistable limit and construct an explicit sequence of blow-ups and contractions which transforms the special fiber into an $m$-stable curve. To show that $m$-stable limits are unique, we consider two $m$-stable curves $\C_{1}/\Delta$ and $\C_{2}/\Delta$ with isomorphic generic fiber, and a semistable curve $\C^{ss}/\Delta$ which dominates both. Using the results of section \ref{S:stablelimits} on semistable tails of the elliptic $m$-fold point, we prove that the exceptional locus of $\phi_{1}: \C^{ss} \rightarrow \C_{1}$ is the equal to the exceptional locus of $\phi_{2}:\C^{ss} \rightarrow \C_{2}$, so $\C_{1} \simeq \C_{2}$ as desired. Finally, in section \ref{S:mAstability}, we explain how to produce the $(m,\A)$-stable limit by starting with the $m$-stable limit and running a relative minimal model program with respect to $\omega_{\C/\Delta}(\sum_{i}a_ip_i)$. This procedure is exactly analagous to Hassett's construction of the $\A$-stable limit starting from the Deligne-Mumford stable limit \cite{Hassett4}.  It would be interesting to have an interpretation of the valuative criterion for $\SM_{1,n}(m)$ as a relative-MMP with respect to a certain line-bundle on the universal curve $\C \rightarrow \SM_{1,n}$, but the author is not aware of such an interpretation.

In appendix A, we prove that the only isolated Gorenstein singularities that can occur on a reduced curve of arithmetic genus one are nodes and elliptic $l$-fold points. The proof is pure commutative algebra: we simply classify all one-dimensional complete local rings with the appropriate numerical invariants. The result plays a crucial simplifying role throughout the paper. Using this classification, for example, one does not need any `serious' deformation theory to see that only nodes and elliptic $l$-fold points, $l \leq m$, can occur as deformations of the elliptic $m$-fold point. Another fact that we use repeatedly is that if one contracts a smooth elliptic curve in the special fiber of a flat family of curves with smooth total space, the image of the elliptic curve in the new special fiber is an elliptic $m$-fold point. Using  Lemma \ref{L:Contraction}, this is an easy consequence of our classification.

\subsection{Notation}\label{S:notation} 
A \emph{curve} is a reduced connected 1-dimensional scheme of finite-type over an algebraically closed field. Starting in section three, all curves will be assumed complete (this assumption is irrelevant in section two, which is essentially a local study). An \emph{$n$-pointed curve} is a curve $C$, together with $n$ smooth marked points $p_1, \ldots, p_n \in C$ (not necessarily distinct). If $(C,p_1, \ldots, p_n)$ is an $n$-pointed curve, and $F \subset C$ is an irreducible component, we say that a point of $F$ is \emph{distinguished} if it is a marked point or a singular point. In addition, if $\tilde{C}$ is the normalization of $C$, we say that a point of $\tilde{C}$ is distinguished if it lies above a marked point or a singular point of $C$. An $n$-pointed curve is \emph{nodal} if it has only nodes as singularities and the marked points are distinct. An $n$-pointed curve is \emph{semistable} (resp. \emph{stable}) if it is nodal and every rational component of $\tilde{C}$ has at least two (resp.) three distinguished points.

A reduced curve is automatically Cohen-Macaulay, and therefore possesses a dualizing sheaf $\omega_{C}$. We say that $C$ is \emph{Gorenstein} if $\omega_{C}$ is invertible. Note that any flat, projective, finitely-presented morphism $X \rightarrow T$ whose geometric fibers are Cohen-Macaulay admits a relative dualizing sheaf $\omega_{X/T}$, whose formation commutes with base-change \cite{Kleiman}. In particular, if the geometric fibers of $X \rightarrow T$ are Gorenstein curves, then $\omega_{X/T}$ is invertible.

$\Delta$ will always denote the spectrum of a discrete valuation ring $R$ with algebraically closed residue field $k$ and field of fractions $K$. When we speak of a finite base-change $\Delta' \rightarrow \Delta$, we mean that $\Delta'$ is the spectrum of a discrete valuation ring $R' \supset R$ with field of fractions $K'$, where $K' \supset K$ is a finite separable extension. We use the notation
\begin{align*}
0&:=\Spec k \rightarrow \Delta,\\
\eta&:=\Spec K \rightarrow \Delta,\\
\overline{\eta}&:=\Spec \overline{K} \rightarrow \Delta,
\end{align*}
for the closed point, generic point, and geometric generic point respectively.
Families over $\Delta$ will be denoted in script, while geometric fibers are denoted in regular font. For example, $C_0, \C_{\eta}, C_{\overline{\eta}}$ and $C'_0, \C'_{\eta}, C'_{\overline{\eta}}$  denote the special fiber, generic fiber, and geometric generic fibers of $\C \rightarrow \Delta$ and $\C' \rightarrow \Delta$ respectively. We will often omit the subscript `0' for the special fiber, and simply write $C, C'.$\\

\textbf{Acknowledgements.} This research was conducted under the supervision of Joe Harris, whose encouragement and insight were invaluable throughout. The problem of investigating the birational geometry of $\M_{1,n}$ was suggested by Brendan Hassett, who invited me to Rice University at a critical juncture and patiently explained his beautiful ideas concerning the log-minimal model program for $\M_{g}$. Finally, I am grateful to Dawei Chen, Maksym Fedorchuk, Matthew Simpson, and Fred van der Wyck for numerous helpful and exciting conversations.

\section{Geometry of the elliptic $m$-fold point}
In this section, we work over a fixed algebraically closed field $k$. We consider a curve $C$ with a singular point $p \in C$, and let $\pi: \tilde{C} \rightarrow C$ denote the normalization of $C$ at $p$. $\hat{\O}_{C,p}$ will denote the completion of the local ring of $C$ at $p$, and $m_p \subset \hat{\O}_{C,p}$ the maximal ideal. In addition, we let  $\pi^{-1}(p)=\{p_1, \ldots, p_m\}$ and set
$$
 \hat{\O}_{\tilde{C},\pi^{-1}(p)} := \oplus_{i=1}^{m}\hat{\O}_{\tilde{C},p_i}.
$$
Note that a choice of uniformizers $t_i \in m_{p_i}$ induces an identification
$$ \hat{\O}_{\tilde{C},\pi^{-1}(p)} \simeq k[[t_1]] \oplus \ldots \oplus k[[t_m]].$$

We will be concerned with the following sequence of singularities:

\begin{definition}[The elliptic $m$-fold point] \label{D:mpoint}We say that $p$ is an \emph{elliptic $m$-fold point} of $C$ if
\begin{align*}
\hat{\O}_{C,p} \simeq &
\begin{cases}
k[[x,y]]/(y^2-x^3) & m=1\,\,\,\text{ (ordinary cusp)}\\
k[[x,y]]/(y^2-yx^2) & m=2 \,\,\,\text{ (ordinary tacnode)} \\
k[[x,y]]/(x^2y-xy^2) & m=3 \,\,\, \text{ (planar triple-point)}\\
k[[x_1, \ldots, x_{m-1}]]/I_m & m \geq 4, \text{ (cone over $m$ general points in $\mathbb{A}^{m-2}$),}\\
\end{cases}\\
&\,\,\,\,\, I_{m}:= \left(  x_{h}x_i-x_{h}x_j \, : \,\, i,j,h \in \{1, \ldots, m-1\} \text{ distinct} \right).
\end{align*}
\end{definition}
One checks immediately that, for an appropriate choice of uniformizers, the map $\pi^*:\hat{\O}_{C,p} \hookrightarrow \hat{\O}_{\tilde{C},\pi^{-1}(p)}$ is given by 
\begin{align*}
\left(
\begin{matrix}
x\\
y
\end{matrix}
\right)&\rightarrow
\left(
\begin{matrix}
t_1^2\\
t_1^3\\
\end{matrix}
\right) &m=1\\
\left(
\begin{matrix}
x\\
y\\
\end{matrix}
\right)&\rightarrow
\left(
\begin{matrix} \tag{\dag}
t_1 & t_1^2\\
t_2 & 0\\
\end{matrix}
\right) &m=2\\
\left(
\begin{matrix}
x_1\\
\vdots\\
\vdots\\
x_{m-1}
\end{matrix}
\right)&\rightarrow
\left(
\begin{matrix}
t_1& 0& \hdots  & 0 & t_{m}\\
0&t_2& \ddots & \vdots & t_{m} \\
\vdots& \ddots& \ddots& 0& \vdots  \\
0 & \hdots &0 & t_{m-1} & t_{m}
\end{matrix}
\right) & m \geq 3
\end{align*}
It will also be useful to have the following coordinate-free characterization of the elliptic $m$-fold point. 
\begin{lemma}\label{L:mpoint}
$p \in C$ is an elliptic $m$-fold point $\iff \pi^*:\hat{\O}_{C,p} \hookrightarrow \hat{\O}_{\tilde{C},\pi^{-1}(p)}$ satisfies
\begin{itemize}
\item[(1)] $\pi^{*}(m_{p}/m_{p}^2) \subset \oplus_{i=1}^{m} m_{p_i}/m_{p_i}^2$
is a codimension-one subspace.
\item[(2)] $\pi^{*}(m_{p}/m_{p}^2) \supsetneq m_{p_i}/m_{p_i}^2 \text{ for any $i=1, \ldots, m$}.$
\item[(3)] $\pi^*(m_{p}^2) = \oplus_{i=1}^{m} m_{p_i}^2.$
\end{itemize}
Furthermore, if $m \geq 3$, then (1) and (2) automatically imply (3).
\end{lemma}
It is useful to think of the lemma as describing when a function $f$ on $\tilde{C}$ descends to $C$. Part (3) says that if $f$ vanishes to order at least two along $p_1, \ldots, p_m$, then $f$ descends to $C$. Part (1) says that if $f$ vanishes at $p_1, \ldots, p_m$, then the derivatives of $f$ at $p_1, \ldots, p_m$ must satisfy one additional condition in order for $f$ to descend.
\begin{proof}
If $p \in C$ is an elliptic $m$-fold point, then one easily checks (1) -(3) using $(\dag)$. Conversely, if $\pi^*$ satisfies (1)-(3), we will show that it is possible to choose coordinates at $p$ and uniformizers at $p_1, \ldots, p_m$ so that the map $\pi^*$ takes the form (\dag). Start by picking any basis $\{x_1, \ldots, x_{m-1}\}$ for the codimension-one subspace
$$\pi^*(m_{p}/m_{p}^2) \subset \oplus_{i=1}^{m}(t_i)/(t_i^2),$$ and write
\[x_{i}=\oplus_{j=1}^{m}a_{ij}t_j, \,\,a_{ij} \in k, \, t_j \in \m_{p_j}.\]
Reordering the branches if necessary, we can use Gaussian elimination to bring the matrix $\{a_{ij}\}$ into the form
\[
\left(
\begin{matrix}
1& 0& \hdots  & 0 & c_1\\
0&1& \ddots & \vdots & c_2\\
\vdots& \ddots& \ddots& 0& \vdots  \\
0 & \hdots &0 & 1 & c_{m-1}
\end{matrix}
\right)
\]
where $c_1, \ldots, c_{m-1} \in k$. Then (2) implies that $c_1, \ldots, c_{m-1} \in k^*$. Thus, if we change uniformizers by setting $t_i'=c_it_i$ for $i=1, \ldots, m-1$ and $t_m'=t_m$, we see that $\pi^*(m_{p}/m_{p}^2)$ is the span of
\begin{align*}\label{func2}
\{(t_1', 0, 0 \ldots, 0, t_m'), (0, t_2', 0 \ldots, 0, t_m')\ldots, (0, \ldots, 0, t_{m-1}', t_{m}')\}.
\end{align*}
This proves the lemma when $m \geq 3.$ We leave the details of $m=1,2$ to the reader.
\end{proof}
\subsection{The tangent sheaf $\Omega^{\vee}_{C}$}\label{S:tangentsheaf}
The tangent sheaf of $C$ and $\tilde{C}$ are defined as
\begin{align*}
\Omega^{\vee}_C:=&\Hom_{\O_{C}}(\Omega_{C},\O_C),\\
\Omega^{\vee}_{\tilde{C}}:=&\Hom_{\O_{\tilde{C}}}(\Omega_{\tilde{C}},\O_{\tilde{C}}),
\end{align*}
respectively. Let $K(\tilde{C})$ denote the constant sheaf of rational functions on $\tilde{C}$. Then we have a natural inclusion
$$\Omega^{\vee}_{C} \hookrightarrow \pi_*\Omega^{\vee}_{\tilde{C}} \otimes K(\tilde{C}),$$
given by restricting a regular vector field on $C$ to $C \backslash \{p\} \simeq \tilde{C} \backslash \{p_1, \ldots, p_m\}$, and then extending to a \emph{rational} section of $\pi_*\Omega^{\vee}_{\tilde{C}}$. If $p$ is an ordinary node then this inclusion induces an isomorphism
$$
\Omega^{\vee}_{C} \simeq \pi_*\Omega^{\vee}_{\tilde{C}}(-p_1-p_2).
$$
In other words, a regular vector field on $\tilde{C}$ decends to $C$ iff it vanishes at the points lying above the node \cite{DM}.

In Proposition \ref{P:tangentsheaf}, we give a similar description of $\Omega^{\vee}_{C}$ when $p \in C$ is an elliptic $m$-fold point. In this case, $\Omega^{\vee}_{C} \subset \pi_*\Omega^{\vee}_{\tilde{C}}$ is precisely the sheaf of regular vector fields on $\tilde{C}$ which vanish at $p_1, \ldots, p_m$, \emph{and} have the same first-derivative at $p_1, \ldots, p_m$. This allows us to say when a curve with an elliptic $m$-fold point has infintesimal automorphisms, and in particular to conclude that $m$-stable curves have none.

Before stating Proposition \ref{P:tangentsheaf}, we pause to highlight a certain positive characteristic pathology. One might hope that, for an arbitrary isolated curve singularity, the inclusion
$$\Omega^{\vee}_C \hookrightarrow \pi_*\Omega^{\vee}_{\tilde{C}} \otimes K(\tilde{C})$$ 
always factors through $\pi_*\Omega^{\vee}_{\tilde{C}}$.
In other words, regular vector fields on $C$ are always induced by regular vector fields on $\tilde{C}$. This is true in characteristic 0, but may fail in characteristic $p>0$ as the following example shows. (We thank Fred van der Wyck for alerting us to this pitfall.)
\begin{example} \label{E:automorphisms}
Suppose that $\characteristic k=2$ and
$$
C:=\Spec k[x,y]/(y^2-x^3).
$$
Then $\frac{d}{dy}$ is a section of a $\Omega^{\vee}_{C}$ since
$$
\frac{d}{dy} (y^2-x^3)=2y=0.
$$
The map $\pi:\tilde{C} \rightarrow C$ is given by
$$t \rightarrow (t^2,t^3),$$
so
\begin{align*}
\pi^*dy&=3t^2dt,\\
\frac{d}{dy}&=\pi_*\left( \frac{1}{3t^2}\frac{d}{dt} \right).
\end{align*}
In other words, $\frac{d}{dy}$ is a vector field on $C$ whose extension to a rational vector field on $\tilde{C}$ carries a double pole. 

\end{example}
It is this pathology that accounts for the restrictions on $\characteristic k$ that occur in the following proposition.
\begin{proposition}[Tangent sheaf of the elliptic $m$-fold point]\label{P:tangentsheaf}
Suppose that one of the following three conditions holds.
\begin{itemize}
\item[(1)] $p$ is a cusp and $\characteristic k \neq 2,3$,
\item[(2)] $p$ is a tacnode and $\characteristic k \neq 2$,
\item[(3)] $p$ is an elliptic $m$-fold point and $m \geq 3$.
\end{itemize}
Consider the exact sequence
$$
0 \rightarrow \pi_*\Omega^{\vee}_{\tilde{C}}(-\Sigma_i2p_i) \rightarrow \pi_*\Omega^{\vee}_{\tilde{C}}(-\Sigma_i p_i) \rightarrow \oplus_{i=1}^{m}\Omega^{\vee}_{\tilde{C}}(-p_i)|_{p_i} \rightarrow 0.
$$
Since we have a canonical isomorphism
$$\oplus_{i=1}^{m}\Omega^{\vee}_{\tilde{C}}(-p_i)|_{p_i} \simeq \oplus_{i=1}^{m}k,$$
there is a well-defined diagonal map
$$\Delta: k \hookrightarrow \oplus_{i=1}^{m}\Omega^{\vee}_{\tilde{C}}(-p_i)|_{p_i},$$
and $\Omega^{\vee}_{C} \subset \pi_*\Omega^{\vee}_{\tilde{C}}$ is the inverse image of $\Delta \subset  \oplus_{i=1}^{m}\Omega^{\vee}_{\tilde{C}}(-p_i)|_{p_i}$. Equivalently, if we let
\begin{align*}
\oplus_{i=1}^{m}f_i(t_i)\frac{d}{dt_i}&\text{ with } f_i=a_{i0}+a_{i1}t_i+g_i \text{, } g_i \in (t_i)^{2}
\end{align*}
be the local expansion of a section of $\Omega^{\vee}_{\tilde{C}}$ around $p_1, \ldots, p_m$, then $\Omega^{\vee}_{C} \subset \pi_*\Omega^{\vee}_{\tilde{C}}$ is the subsheaf generated by those sections which satisfy
\begin{align*}
&a_{10}=\ldots=a_{m0}=0,\\
&a_{11}=\ldots=a_{m1}. 
\end{align*}
\end{proposition}
\begin{proof}
A section of $\Omega^{\vee}_{\tilde{C}} \otimes K(\tilde{C})$ is contained in $\Omega^{\vee}_{C}$ iff its image under the push-forward map
$$\pi_*: \pi_*\Hom(\Omega_{\tilde{C}},K(\tilde{C})) \rightarrow \Hom(\Omega_{C},\pi_*K(\tilde{C})),$$
lies in the subspace
$$
\Hom(\Omega_{C},\O_{C}) \subset \Hom(\Omega_{C},\pi_*K(\tilde{C})).
$$
Thus, we must write out the push-forward map in local coordinates. We may work formally around $p$ and use the coordinates introduced in $(\dag)$. \\
\begin{itemize}
\item[(1)](The cusp) The section $f(t_1)\frac{d}{dt_1} \in \pi_*\Omega^{\vee}_{\tilde{C}} \otimes K(\tilde{C})$ pushes forward to
$$
\pi_*\left(f(t_1)\frac{d}{dt_1}\right)=2t_1f(t_1) \frac{d}{dx}+3t_1^2f(t_1)\frac{d}{dy}.
$$
Since $\hat{\O}_{C,p}=k[[t_1^2,t_1^3]] \subset k[[t_1]]$, we see that if $\characteristic k \neq 2,3,$ then
$$2t_1f(t_1), 3t^2f(t_1) \in \hat{\O}_{C,p} \iff f(t_1) \in (t_1).$$
Thus, 
$$\Omega^{\vee}_{C}=\pi_*\Omega^{\vee}_{\tilde{C}}(-p_1).$$
\\
\item[(2)](The tacnode)
The section $f_1(t_1)\frac{d}{dt_1} \oplus f_2(t_2)\frac{d}{dt_2} \in \pi_*\Omega^{\vee}_{\tilde{C}} \otimes K(\tilde{C})$
pushes forward to
$$
\pi_*\left(f_1(t_1)\frac{d}{dt_1}\oplus f_2(t_2)\frac{d}{dt_2} \right)=(f_1(t_1)\oplus f_2(t_2)) \frac{d}{dx}+(2t_1f_1(t_1)\oplus 0)\frac{d}{dy}.
$$
If $\characteristic k \neq 2$, then
$$
(2t_1f_1(t_1)\oplus 0) \in \hat{\O}_{C,p} \iff f_1(t_1) \in (t_1).
$$
Furthermore, once we know $f_1(t_1) \in (t_1)$, then
$$
(f_1(t_1)\oplus f_2(t_2)) \in \hat{\O}_{C,p} \iff f_1(t_1)\oplus f_2(t_2)=a(t_1\oplus t_2)+(g_1 \oplus g_2)
$$
 for some $a \in k$ and $(g_1\oplus g_2) \in (t_1^2) \oplus (t_2^2)$, which is precisely the conclusion of the proposition.\\
\item[(3)]($m \geq 3$)
The section $\oplus_{i=1}^{m}f_i(t_i)\frac{d}{dt_i} \in \pi_*\Omega^{\vee}_{\tilde{C}} \otimes K(\tilde{C})$
pushes forward to
$$
\pi_*\left( \oplus_{i=1}^{m}f_i(t_i)\frac{d}{dt_i} \right)=\sum_{i=1}^{m-1}(f_i(t_i) \oplus f_m(t_m)) \frac{d}{dx_i}
$$
Note that the function $(f_i(t_i) \oplus f_m(t_m))$ vanishes identically on all branches except the $i^{th}$ and $m^{th}$. It follows that, for each $i=1, \ldots, m-1$,
$$
(f_i(t_i) \oplus f_m(t_m)) \in \hat{\O}_{C,p} \iff (f_i(t_i) \oplus f_m(t_m))=a(t_i \oplus t_m) + (g_1 \oplus g_2),
$$
for some $a \in k$ and $(g_i \oplus g_m) \in (t_i^2) \oplus (t_m^2)$. Thus,
$$
\oplus_{i=1}^{m}f_i(t_i)=a(t_1 \oplus \ldots \oplus t_m)+(g_1(t_1) \oplus \ldots \oplus g_m(t_m)),
$$
for some $a \in k$ and $g_i \in (t_i^2)$. This completes the proof of the proposition.
\end{itemize}
\end{proof}
Our only use for Proposition \ref{P:tangentsheaf} is the following corollary, which translates the condition of having no infinitesimal automorphisms into a condition on distinguished points.
\begin{corollary}\label{C:automorphisms} Suppose $\characteristic k \neq 2,3$. Let $C$ be a complete $n$-pointed curve $(C, q_1, \ldots, q_n)$, and let $\Sigma$ denote the support of the divisor $\sum_{i}q_i$. Suppose $C$ has an elliptic $m$-fold point $p \in C$, and that the normalization of $C$ at $p$ consists of $m$ distinct connected components:
$$
\tilde{C}=\tilde{C}_1 \cup \ldots \cup \tilde{C}_m,
$$
where each $\tilde{C}_i$ is a nodal curve of arithmetic genus zero. Then we have:
$$
H^0(C,\Omega^{\vee}_C(-\Sigma))=0 \iff \text{conditions (1), (2), and (3) hold.}
$$
\begin{itemize}
\item[(1)] $\tilde{B}_1, \ldots, \tilde{B}_m$ each have $\geq 2$ distinguished points, where $\tilde{B}_i \subset \tilde{C}_i$ is the unique irreducible component of $\tilde{C}_i$ lying above $p$,
\item[(2)] At least one of $\tilde{B}_1, \ldots, \tilde{B}_m$ has $\geq 3$ distinguished points,
\item[(3)] Every other component of $\tilde{C}$ has $\geq 3$ distinguished points.
\end{itemize}
\end{corollary}
\begin{proof}
First, let us check that these conditions are necessary. For (1), suppose that $\tilde{B}_i$ has only one distinguished point. Then this distinguished point is necessarily $p_i$, the point lying above $p$, so $\tilde{B}_i=\tilde{C}_i$, and $\tilde{C}_i$ has a non-zero vector field which vanishes to order two at $p_i$. One may extend this section (by zero) to a section of $\Omega^{\vee}_{\tilde{C}}(-\Sigma)$, and Proposition \ref{P:tangentsheaf} implies that it decends to give a non-zero section of $\Omega^{\vee}_{C}(-\Sigma)$.

For (2), suppose that each $\tilde{B}_i$ has exactly two distinguished points, say $p_i$ and $r_i$. Then the restriction map
$$
\oplus_{i=1}^{m}\Omega^{\vee}_{\tilde{B}_i}(-p_i-r_i) \rightarrow \oplus_{i=1}^{m}\Omega^{\vee}_{\tilde{B}_i}(-p_i-r_i)|_{p_i} \rightarrow 0
$$
is surjective on global sections. Thus we can find sections of $\Omega^{\vee}_{\tilde{B}_i}$ which vanish at $p_i$ and $r_i$, and whose first derivatives at $p_1, \ldots, p_m$ agree. We can extend these (by zero) to a section of $\Omega^{\vee}_{\tilde{C}}(-\Sigma)$, and Proposition \ref{P:tangentsheaf} implies that this descends to give a non-vanishing section of $\Omega^{\vee}_{C}(-\Sigma)$.

Finally, if any other component of $\tilde{C}$ has less than three distinguished points, then there exists a vector field on that component which vanishes at all distinguished points. Since this component necessarily meets the rest of $\tilde{C}$ nodally, such a section can be extended (by zero) to a section of $\Omega^{\vee}_{\tilde{C}}(-\Sigma)$ which descends to $\Omega^{\vee}_{C}(-\Sigma)$.

Now let us check that conditions $(1)$, $(2)$, and $(3)$ are sufficient. One easily checks that conditions (1) and (3) imply
$$H^0(\tilde{C}_1,\Omega^{\vee}_{\tilde{C}_1}(-2p_2-\Sigma|_{\tilde{C}_1}))=\ldots =H^0(\tilde{C}_m, \Omega^{\vee}_{\tilde{C}_m}(-2p_m-\Sigma|_{\tilde{C}_m}))=0,$$ while conditions (2) and (3) imply that, for some $i$, we have
$$
H^0(\tilde{C}_i,\Omega^{\vee}_{\tilde{C}_i}(-p_i-\Sigma|_{\tilde{C}_i}))=0.
$$
This latter condition says that any section of $\Omega^{\vee}_{\tilde{C}_i}$ which vanishes at $p_i$ must vanish identically. It follows, by Proposition \ref{P:tangentsheaf}, that any section of $\Omega^{\vee}_{\tilde{C}}(-\Sigma)$ which decends to a section of $\Omega^{\vee}_{C}(-\Sigma)$ must vanish at $p_1, \ldots, p_m$ \emph{and} have vanishing first-derivative at $p_1, \ldots, p_m$. But since
$$H^0(\tilde{C}_1,\Omega^{\vee}_{\tilde{C}_1}(-2p_2-\Sigma|_{\tilde{C}_1}))=\ldots =H^0(\tilde{C}_m, \Omega^{\vee}_{\tilde{C}_m}(-2p_m-\Sigma|_{\tilde{C}_m}))=0,$$
any section of $\Omega^{\vee}_{\tilde{C}}$ satisfying these conditions is identically zero.

\end{proof}

\subsection{The dualizing sheaf $\omega_{C}$}\label{S:dualizingsheaf}
In the following proposition, we describe the dualizing sheaf $\omega_{C}$ locally around an elliptic $m$-fold point. If $p \in C$ is a singular point on a reduced curve, then, locally around $p$, $\omega_{C}$ admits the following explicit description: Let $\pi: \tilde{C} \rightarrow C$ be the normalization of $C$ at $p$ and consider the sheaf $\Omega_{\tilde{C}} \otimes K(\tilde{C})$ of rational differentials on $\tilde{C}$. Let $K_{\tilde{C}}(\Delta) \subset \Omega_{\tilde{C}} \otimes K(\tilde{C})$ be the subsheaf of rational differentials $\omega$ satisfying the following condition: For every function $f \in \O_{C,p}$
$$
\sum_{p_i \in \pi^{-1}(p)}\Res_{p_i}((\pi^*f) \,\omega)=0.
$$
Then, locally around $p$, we have $\omega_{C}=\pi_*K_{C}(\Delta).$ (See \cite{Serre} for a general discussion of duality on curves.) Using this description, we can show that
\begin{proposition}\label{P:dualizingsheaf} If $p \in C$ is an elliptic $m$-fold point, then
\begin{itemize}
\item[(1)] $\omega_{C}$ is invertible near $p$, i.e. the elliptic $m$-fold point is \emph{Gorenstein}.
\item[(2)] $\pi^*\omega_{C}=\omega_{\tilde{C}}(2p_1+ \ldots + 2p_m).$
\end{itemize}
\end{proposition}
\begin{proof} We will prove the proposition when $m \geq 3$ and leave the details of $m=1,2$ to the reader. By the previous discussion, sections of $\omega_{C}$ near $p$ are given by rational sections $\omega \in \omega_{\tilde{C}} \otimes K(\tilde{C})$ satisfying
\begin{align*}
 \sum_{i=1}^{m}\Res_{p_i}((\pi^*f)\,\omega)=0 \text{ for all $f \in \O_{C,p}$}.
\end{align*}
By Lemma \ref{L:mpoint} (3), every function vanishing to order $\geq 2$ along $p_1, \ldots, p_m$ descends to $C$, so any differential $\omega$ which satisfies this condition can have at most double poles along $p_1, \ldots, p_m$. Now consider the polar part of $\omega$ around $p_1, \ldots, p_m$, i.e. write
$$
\omega-\omega'=\left(a_1\frac{dt_1}{t_1^2}+b_1\frac{dt_1}{t_1} \right) \oplus  \ldots \oplus \left( a_m\frac{dt_m}{t_m^2} +b_m\frac{dt_m}{t_m} \right)
$$
with $a_i, b_i \in k$ and $\omega' \in \omega_{\tilde{C}}$. It suffices to check the residue condition for scalars and a basis of $m_{p}/m_{p}^2$. Taking $f \in \O_{C}$ to be non-zero scalar, the residue condition gives
$$
b_1+\ldots+b_m=0.
$$
Working in the coordinates $(\dag)$ introduced at the beginning of the section, we see that
$$
\{(t_1, 0, \ldots, 0, t_m), (0,t_2,0, \ldots, 0, t_m), \ldots, (0, \ldots, 0, t_{m-1},t_{m})\}
$$
gives a basis for $m_{p}/m_{p}^2$, so the residue condition forces
\begin{align*}
&&a_i-a_m=0, \text{ for $i=1, \ldots, m-1$.}
\end{align*}

From this, one checks immediately that
$$
\left\{ \left(\frac{dt_1}{t_1^2}+ \ldots + \frac{dt_{m-1}}{t_{m-1}^2} -\frac{dt_m}{t_m^2}\right), \left(\frac{dt_{1}}{t_1} -\frac{dt_m}{t_m}\right),\left(\frac{dt_{2}}{t_2} -\frac{dt_m}{t_m}\right), \ldots, \left( \frac{dt_{t_{m-1}}}{t_{m-1}}  -\frac{dt_m}{t_m}\right) \right\}
$$
gives a basis of sections for
$\omega_{C}/\pi_*\omega_{\tilde{C}}$. It follows that multiplication by $$\frac{dt_1}{t_1^2}+ \ldots + \frac{dt_{m-1}}{t_{m-1}^2}-\frac{dt_m}{t_m^2}$$ gives an isomorphism
$$
\O_{C} \simeq \omega_{C},
$$
so $\omega_{C}$ is invertible. Since a local generator for $\omega_{C}$ pulls back to a differential with a double pole along each of $p_1, \ldots, p_m$,  we also have $$\pi^*\omega_{C}=\omega_{\tilde{C}}(2p_1+ \ldots +2p_m).$$
\end{proof}
\subsection{Semistable limits}\label{S:stablelimits}
Our aim in this section is to classify those `tails' that occur when performing semistable reduction to a one-parameter smoothing of the elliptic $m$-fold point. This will be the key ingredient in the verification of the valuative criterion for $\SM_{1,\A}(m)$. Throughout the section, $C$ denotes a connected curve (not necessarily complete), and for simplicity we will assume that $C$ has a unique singular point $p$. 

\begin{definition}
A \emph{smoothing of $(C,p)$} consists of a morphism $\pi:\C \rightarrow \Delta$, where $\Delta$ is the spectrum of a discrete valuation ring with residue field $k$, and a distinguished closed point $p \in \C$ satisfying
\begin{itemize}
\item[(1)] $\pi$ is quasiprojective and flat of relative dimension 1.
\item[(2)] $\pi$ is smooth on $U:=\C \backslash p$.
\item[(3)] The special fiber of $\pi$ is isomorphic to $(C,p).$
\end{itemize}
\end{definition}
\begin{definition}
If $\C/\Delta$ is a smoothing of $(C,p)$, a \emph{semistable limit} of $\C / \Delta$ consists of a finite base-change $\Delta' \rightarrow \Delta$, and a diagram
\[
\xymatrix{
\C^{s} \ar[rr]^{\phi}\ar[rd]_{\pi^{s}}&&\C \times_{\Delta} \Delta'\ar[ld]\\
&\Delta'&
}
\]
satisfying
\begin{itemize}
 \item[(1)] $\pi^s$ is quasiprojective and flat of relative dimension 1.
 \item[(2)] $\phi$ is proper, birational, and $\phi(\Exc(\phi))=p$.
 \item[(3)] The total space $\C^s$ is regular, and the special fiber $C^s$ is nodal.
\item[(4)] $\Exc(\phi)$ contains no (-1)-curves.
\end{itemize}
The \emph{exceptional curve} of the semistable limit is the pair $(E,\Sigma)$ where
\begin{align*}
E&:=\phi^{-1}(p),\\
\Sigma&:=\{ E \cap \overline{C^s \backslash E} \}.
\end{align*}
We think of $\Sigma$ as a reduced effective Weil divisor on $E$. Note that $(E, \Sigma)$ is necessarily  semistable, i.e. nodal and each rational component of $\tilde{E}$ has two distinguished points.
\end{definition}
\begin{remarks}
\begin{itemize}
\item[]
\item[(1)] If $\C/\Delta$ is any smoothing of $p$, then the total space of $\C$ is normal by Serre's criterion. In particular, $\Exc(\phi)$ is connected by Zariski's main theorem.
\item[(2)] Semistable limits exist: If $\C/\Delta$ is a smoothing of $p$, let $\tilde{\C}$ denote the normalization of the closure of $\C$ under some projective embedding (over $\Delta$). Then $\tilde{\C} \rightarrow \Delta$ will be proper, flat of relative dimension 1, and smooth over the generic fiber. Furthermore, by the previous remark, there exists an open immersion $\C \hookrightarrow \tilde{\C}$. Applying semistable reduction to $\tilde{\C}$, one obtains a finite base-change $\Delta' \rightarrow \Delta$, a nodal family $\tilde{\C}^{s}/\Delta'$, and a birational map $\phi:\tilde{\C}^{s} \rightarrow  \tilde{\C} \times_{\Delta} \Delta'$ \cite{DM}. Restricting $\phi$ to the open subscheme $\phi^{-1}(\C \times_{\Delta} \Delta')$ and blowing down any (-1)-curves in $\phi^{-1}(p)$ gives the desired semistable limit.
\item[(3)] Semistable limits are not unique: If $\C'/\Delta'$ is a semistable limit for $\C/\Delta$, and $\Delta'' \rightarrow \Delta'$ is any finite base-change, then taking the minimal resolution of singularities of $\C' \times_{\Delta'} \Delta''$ gives another semistable limit.
\end{itemize}
\end{remarks}

\begin{definition}\label{D:tail}
We say that a pointed curve $(E,p_1, \ldots, p_m)$ is a \emph{semistable tail} of $(C,p)$ if it arises as the exceptional curve of a semistable limit of a smoothing of $(C,p)$. 
\end{definition}

In Proposition \ref{P:semistablelimits}, we classify the semistable tails of the elliptic $m$-fold point. In order to state the result, we need a few easy facts about the dual graph of a nodal curve of arithmetic genus one. (Note that these remarks will be generalized to arbitrary Gorenstein curves of arithmetic genus one in section \ref{S:mstable}.) First, observe that if $E$ is any complete, connected, nodal curve of arithmetic genus one, then $E$ contains a connected, arithmetic genus one subcurve $Z \subset E$ with no disconnecting nodes. (If $E$ itself has no disconnecting nodes, there is nothing to prove. If $E$ has a disconnecting node $q$, then the normalization of $E$ at $q$ will comprise two connected components, one of which has arithmetic genus one. Proceed by induction on the number of disconnecting nodes.) There are two possibilities for $Z$: either it is irreducible or a ring of $\P^1$'s. By genus considerations, the connected components of $\overline{E \backslash Z}$ will each have arithmetic genus zero and will meet $Z$ in a unique point. We record these observations in the following definition.
\begin{definition}\label{D:fundecomp} If $E$ is a connected, complete, nodal curve of arithmetic genus one, there exists a decomposition
\begin{align*}
&E:=Z \cup R_1 \cup \ldots \cup R_m,\\
\end{align*}
where $Z$ is either irreducible or a ring of $\P^1$'s, and each $R_i$ has arithmetic genus zero and meets $Z$ in exactly one point. We call $Z$ the \emph{minimal elliptic subcurve} of $E$.
\end{definition}
Next, we must introduce notation to talk about the distance between various irreducible components of $E$.

\begin{definition}
If $F_1, F_2 \subset E$ are subcurves of $E$, we define $l(F_1,F_2)$ to be the minimum length of any path in the dual graph of $E$ that connects an irreducible component of $F_1$ to an irreducible component of $F_2$. If $p \in E$ is any smooth point, then there is a unique irreducible component $F_p \subset E$ containg $p$, and we abuse notation by writing write $l(p,-)$ instead of $l(F_p,-)$.
\end{definition}

Now we can state the main result of this section.
\begin{proposition}[Semistable tails of the elliptic $m$-fold point]\label{P:semistablelimits} Suppose $p \in C$ is an elliptic $m$-fold point. If $(E,p_1, \ldots, p_m)$ is a semistable pointed curve of arithmetic genus one, then $(E,p_1, \ldots, p_m)$ is a semistable tail of $(C,p)$ iff 
$$
l(Z,p_1)=l(Z,p_2)=\ldots=l(Z,p_m),
$$
where $Z \subset E$ is the minimal elliptic subcurve of $E$. If $(E,p_1, \ldots, p_m)$  satisfies this condition, we say that it is \emph{balanced}.
\end{proposition}
The proof of this statement is fairly involved. To get a feeling for why it should be true, let us consider some simple examples. First, suppose $E$ is an irreducible curve of arithmetic genus one. Then $Z=E$ and $l(Z,p_i)=0$ for all $i$, so the condition of being balanced is vacuous. In other words, every irreducible pointed elliptic curve $(E,p_1, \ldots, p_m)$ arises as the semistable limit of an elliptic $m$-fold point. To see this explicitly, just attach (nodally) an arbitrary curve $\tilde{C}$ to $E$ along $p_1, \ldots, p_m$, and smooth the curve $\tilde{C} \cup E$ to a family $\C/\Delta$ with smooth total space. Then one may consider the contraction associated to a high power of $\omega_{\C/\Delta}(E)$; using Lemma \ref{L:Contraction}, one can check that this contraction replaces $E$ by an elliptic $m$-fold point, and thus exhibits $(E,p_1, \ldots, p_m)$ as a semistable tail for the elliptic $m$-fold point.

For a second example, suppose that $E=Z \cup R_1 \cup \ldots \cup R_m$, where $Z$ is a smooth elliptic curve, $R_1, \ldots, R_m$ are chains of $\P^1$'s, and the marked point $p_i$ lies at the end of the chain $R_i$. Then $(E,p_1, \ldots, p_m)$ is balanced iff each chain $R_i$ has the same number of components. Why should this be a necessary condition for $(E,p_1, \ldots, p_m)$ to be a semistable tail of $(C,p)$? Well, if $\C/\Delta$ is a smoothing of $(C,p)$, and
\[
\xymatrix{
\C^{ss} \ar[rr]^{\phi} \ar[rd]&&\C \ar[ld]\\
&\Delta&}
\]
is a birational contraction from a semistable curve with exceptional curve $(E,p_1,\ldots, p_m)$, then $\phi$ factors through the stable reduction of $\C^{ss}$, i.e. the birational morphism $\C^{ss} \rightarrow \C^{s}$ obtained by blowing down the chains $R_1, \ldots, R_m.$ The exceptional locus of $\C^{s} \rightarrow \C$ is now an irreducible elliptic curve $Z$, but the total space $\C^{s}$ has singularities of type $(xy-t^{l(p_i,Z)})$ at the points $C^{s} \cap \overline{C^{s} \backslash Z}$. The key observation is that, since $\C \rightarrow \Delta$ has Gorenstein fibers, $\omega_{\C/\Delta}$ is invertible and we must have
$$
\phi^*\omega_{\C/\Delta}=\omega_{\C^{s}/\Delta}(D),
$$
where $D$ is a Cartier divisor supported on $Z$. Since $Z$ is irreducible, this means $D=nZ$ for some $n\,|\,\text{lcm}_{i}\{l(p_i,Z)\}$. Furthermore, since $\phi$ contracts $Z$, we must have $\omega_{\C/\Delta}(D)|_Z \simeq \O_{Z}$. One easily sees that this is possible if and only if $l(p_1,Z)=\ldots=l(p_m,Z)=l$ and $D=lZ$.

The proof of Proposition \ref{P:semistablelimits}  generalizes this idea to the case where $R_1, \ldots, R_m$ are trees of arbitrary combinatorial type. Thus, the true content of Proposition \ref{P:semistablelimits} is that the only obstruction to $(E,p_1, \ldots, p_m)$ being a semistable tail for an elliptic $m$-fold point comes from the necessity of being able to build a line-bundle of the form
$ \omega_{\C^{ss}/\Delta}(D),$ with $\Supp D \subset E$ and $\omega_{\C/\Delta}(D)|_{E} \simeq \O_{E}$, on some semistable curve $\C^{ss}/\Delta$ containing $E$ in the special fiber.

To prove Proposition \ref{P:semistablelimits}, we need the following lemma. In conjunction with the classification of singularities in appendix A, it tells us that whenever we contract an elliptic curve $E$ in the special fiber of a 1-parameter family, \emph{using a line-bundle of the form $\omega_{\C/\Delta}(D)$} with $\Supp D \subset E$, then the resulting special fiber has an elliptic $m$-fold point. Without using a line-bundle of this special form, one cannot guarantee that the resulting curve singularity is Gorenstein.

\begin{lemma}[Contraction lemma]\label{L:Contraction}
Let $\pi:\C \rightarrow \Delta$ be projective and flat of relative dimension one, with smooth general fiber and connected reduced special fiber. Let $\L$ be a line-bundle on $\C$ with positive degree on the generic fiber and non-negative degree on each irreducible component of the special fiber. Set
$$
E = \{ \text{\emph{Irreducible components }} F \subset C \,|\, \deg\L|_{F}=0 \},
$$
and assume that
\begin{itemize} 
\item[(1)]$E$ is connected of arithmetic genus one,
\item[(2)] $\L|_{E} \simeq \O_{E}$,
\item[(3)]Each point $p \in \overline{C \backslash E} \cap E$ is a node of $C$,
\item[(4)]Each point $p \in \overline{C \backslash E} \cap E$ is a regular point of $\C$.
\end{itemize}
Then $\L$ is $\pi$-semiample and there exists a diagram
\[
\xymatrix{\C \ar[rr]^{\phi} \ar[rd]_{\pi} &&\C' \ar[ld]^{\pi'} &\!\!\!\!\!\!\!\!\!\!\!\!\!\!\!\!:=\Proj \left( \oplus_{m \geq 0} \pi_*\L^m \right)\\
&\Delta&&
}
\]
where $\phi$ is proper, birational, and $\Exc(\phi)=E$. Furthermore, we have
\begin{itemize}
\item[(1)] $\C'/\Delta$ is flat and projective, with connected reduced special fiber.
\item[(2)] $\phi|_{\overline{C/E}}:\overline{C/E} \rightarrow C'$ is the normalization of $C'$ at $p:=\phi(E)$,
\item[(3)] The number of branches and the $\delta$-invariant of the isolated curve singularity $p \in C'$ are given by 
\begin{align*}
m&=|\overline{C \backslash E} \cap E|\\
\delta&=p_a(E)+m-1.
\end{align*}
\end{itemize}
If, in addition, we assume that $\omega_{\C/\Delta}$ is invertible and that
$$\L \simeq \omega_{\C/\Delta}(D+\Sigma),$$ where $D$ is a Cartier divisor supported on $E$, and $\Sigma$ is a Cartier divisor disjoint from $E$, then
\begin{itemize}
\item[(4)] $\omega_{\C'/\Delta}$ is invertible. Equivalently, $p \in C'$ is a Gorenstein curve singularity.
\end{itemize}
\end{lemma}
\begin{proof}
To prove that $\L$ is $\pi$-semiample, we must show that the natural map
$$\pi^*\pi_*\L^m \rightarrow \L^m$$
is surjective for $m>>0$. Since $\L$ is ample on the general fiber of $\pi$, it suffices to prove that for each point $x \in C$ there exists a section
$$
s_x \in \pi_*\L^{m}|_0 \subset H^0(C,L^m)
$$
which is non-zero at $x$. Our assumptions imply that $E$ is a Cartier divisor on $\C$, so we have an exact sequence
$$
0 \rightarrow \L^{m}(-E) \rightarrow \L^{m} \rightarrow \O_{E} \rightarrow 0.
$$
Pushing-forward, we obtain
$$
0 \rightarrow \pi_*\L^{m}(-E) \rightarrow \pi_*\L^m \rightarrow \pi_*\O_{E} \rightarrow R^1\pi_*\L^m(-E),
$$
and we claim that $R^1\pi_*\L^m(-E)=0$ for $m>>0$. Since $\L^{m}(-E)$ is flat over $\Delta$, it is enough to prove that this line-bundle has vanishing $H^1$ on fibers for $m>>0$. Since $\L$ is ample on the generic fiber, we only need  to consider the special fiber, where we have an exact sequence
$$0 \rightarrow L^{m} \otimes I_{E/C} \rightarrow \L^{m}(-E)|_{C} \rightarrow \O_{E}(-E) \rightarrow 0$$
We have $H^1(E, \O_{E}(-E))=0$, since $E^{2}<0$ and $E$ has arithmetic genus one. On the other hand, since $L^{m} \otimes I_{E/C}$ is supported on $\overline{C \backslash E}$, we have
$$
H^1(C, L^{m} \otimes I_{E/C})=H^1(\overline{C \backslash E}, (L^{m} \otimes I_{E/C})|_{\overline{C \backslash E}})=0
$$
for $m>>0$, since $L|_{\overline{C \backslash E}}$ is ample. Thus, $H^1(C,\L^{m}(-E)|_{C})=0$. This vanishing has two consequences: First, we have a surjection
$$
\pi_*\L^{m}|_{0} \rightarrow \pi_*\O_{E}|_{0} \simeq k,
$$
so there exists a section $s \in \pi_*\L^{m}|_{0}$ which is constant and non-zero along $E$. Second, we have
$$
\pi_*\L^{m}(-E)|_{0}=H^0(C,L^{m}\otimes I_{E/C}) \subset \pi_*\L^{m}|_{0},
$$
which implies the existence of non-vanishing sections at any point of $\overline{C\backslash E}$.

Since $\L$ is $\pi$-semiample, we obtain a proper, birational contraction $\phi: \C \rightarrow \C'$ with $\Exc(\phi)=E$ and $\phi_*\O_{\C}=\O_{\C'}$.  Since $\C$ is normal, $\C'$ is as well. In particular, $\C'$ is Cohen-Macaulay. The special fiber $C'$ is a Cartier divisor in $\C'$, and hence has no embedded points. No component of $C'$ can be generically non-reduced because it is the birational image of some component of $\overline{C \backslash E}$. Thus, $C'$ is reduced. $C'$ is connected because it is the continuous image of $C,$ which is connected. Finally, since $\C'$ is integral and $\Delta$ is a discrete valuation ring, the flatness of $\pi'$ is automatic. This proves (1).

Conclusion (2) is immediate from the observation that $\overline{C \backslash E}$ is smooth along the points $E \cap \overline{C \backslash E}$ and maps isomorphically to $C'$ elsewhere. Since the number of branches of the singular point $p \in C'$ is, by definition, the number of points lying above $p$ in the normalization, we have
$$m=|\overline{C \backslash E} \cap E|.$$
To obtain $\delta=p_a(E)+m-1$, note that
\begin{align*}
\delta&=\chi(C,\O_{\overline{C \backslash E} })-\chi(C',\O_{C'})\\
&=\chi(C,\O_{\overline{C \backslash E} })-\chi(C,\O_{C})\\
&=-\chi(C,I_{\overline{C \backslash E}}) .
\end{align*}
The first equality is just the definition of $\delta$ since $\overline{C \backslash E}$ is the normalization of $C'$ at $p$. The second equality follows from the fact that $C$ and $C'$ occur in flat families with the same generic fiber, and the third equality is just the additivity of Euler characteristic on exact sequences. Since $I_{\overline{C \backslash E}}$ is supported on $E$, we have
$$\chi(C,I_{\overline{C \backslash E}})=\chi(E,I_{\overline{C \backslash E}}|_{E})= \chi(E,\O_{E}(-E \cap \overline{C \backslash E}))=1-m-p_a(E).$$ 
This completes the proof of (3)

Finally, to prove (4), note that we have a line-bundle $\O_{\C'}(1)$ such that
$$
\phi^*\O_{\C'}(1) \simeq \omega_{\C/\Delta}(D+\Sigma).
$$
Since $\Sigma$ is a Cartier divisor on $\C$ disjoint from $\Exc(\phi)$, its image is a Cartier divisor on $\C'$, and we have
$$
\phi^*\left( \O_{\C'}(1)(-\Sigma) \right) \simeq \omega_{\C/\Delta}(D).
$$
Since $D$ is supported on $\Exc(\phi)$, we have
$$
\O_{\C'}(1)(-\Sigma)|_{\C' \backslash p} \simeq \omega_{\C'/\Delta}|_{\C' \backslash p}.
$$
Since $\omega_{\C'/\Delta}$ and $\O_{\C'}(1)(-\Sigma)$ are both $S_2$-sheaves on a normal surface and they are isomorphic in codimension one, we conclude
$$
\O_{\C'}(1)(-\Sigma) \simeq \omega_{\C'/\Delta},
$$
 i.e. the dualizing sheaf $\omega_{\C'/\Delta}$ is actually a line-bundle. Since the formation of the dualizing sheaf commutes with base-change, $\omega_{C}=\omega_{\C'/\Delta}|_{C}$ is invertible. Equivalently, $p \in C'$ is a Gorenstein singularity.
\end{proof}

Now we are ready to prove Proposition \ref{P:semistablelimits}.
\begin{figure}\label{F:balancedcurve}
\scalebox{.60}{\includegraphics{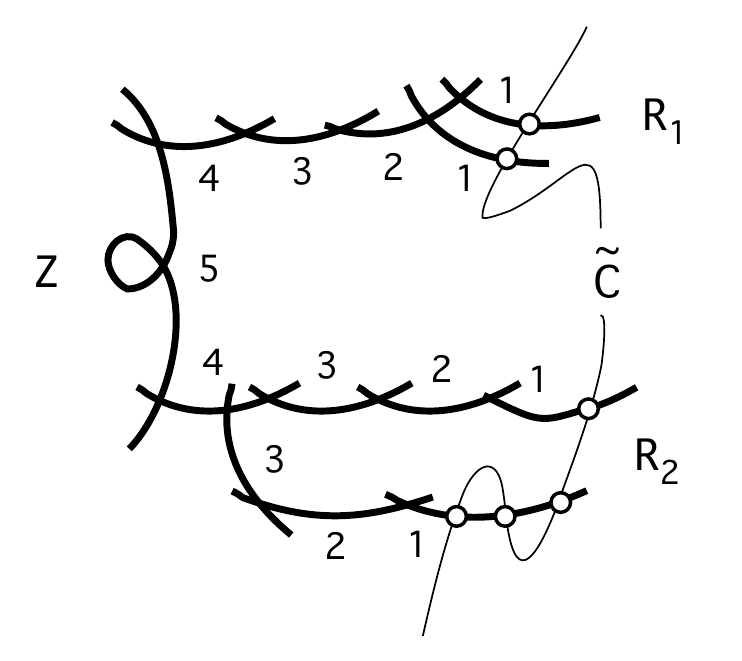}}
\caption{A balanced curve $E$, with minimal elliptic subcurve $Z$, appearing in the special fiber of a semistable family. We have labeled the multiplicities of a Cartier divisor $D$ such that $\omega_{\C/\Delta}(D)$ is trivial on every component of $E$.}\label{F:balancedcurve}
\end{figure}
\begin{proof}[$(E,p_1, \ldots,p_m)$ balanced $\implies (E,p_1, \ldots,p_m)$ is a semistable tail]
We must show that if $(E,p_1, \ldots, p_m)$ is a balanced semistable curve, then it arises as a semistable tail for some smoothing of the elliptic $m$-fold point. To construct this smoothing, start by taking $(\tilde{C}, p_1, \ldots, p_m)$ to be any complete smooth $m$-pointed curve of genus at least two, and attach $\tilde{C}$ and $E$ along $\{p_1, \ldots, p_m\}$ to form  a nodal curve
$$
C^{s}=\tilde{C} \cup E.
$$
Now let $\C^{s}/\Delta$ be any smoothing of $C^{s}$ with smooth total space. We will exhibit a birational morphism $\C^{s} \rightarrow \C$ collapsing $E$ to an elliptic $m$-fold point $p \in C$. To do this, we must build a line-bundle on $\C^{s}$ which is trivial on $E$, but has positive degree on $\tilde{C}$. We define
$$
\L:=\omega_{\C^s/\Delta}(D),
$$
where
$$
D=\sum_{F \subset E} (l+1-l(F,Z))F,
$$
with $l:=l(p_1,Z)=\ldots=l(p_m,Z)$. The multiplicities of $D$ are illustrated in figure \ref{F:balancedcurve}. If we can show that
\begin{align*}
(A)\,\,\,&\L \text{ has positive degree on the general fiber of $\pi$},\\
(B)\,\,\,&\L|_{\tilde{C}} \text{ has positive degree on $\tilde{C}$},\\
(C)\,\,\,&\L|_{F} \simeq \O_{F} \text{ for all irreducible components $F \subset E$},
\end{align*}
then $\L$ satisfies the hypotheses of Lemma \ref{L:Contraction}, so a suitably high multiple of $\L$ defines a morphism
$\phi: \C^{s} \rightarrow \C$
contracting $E$ to a single point $p$. Furthermore, the lemma implies that $p \in C$ is a Gorenstein singularity with $m$ branches and $\delta$-invariant $m$. By Proposition \ref{P:genus1}, there is a unique such singularity: $p$ must be an elliptic $m$-fold point. It follows that $\C/\Delta$ is a smoothing of the elliptic $m$-fold point, and hence that $(E,p_1, \ldots, p_m)$ is a semistable tail as desired.

Since the genus of $\tilde{C}$ is at least two, conditions $(A)$ and $(B)$ are automatic. For condition $(C)$, we write
$$
E=Z \cup R_1 \cup \ldots \cup R_m
$$ 
as in Definition \ref{D:fundecomp}, and consider the cases $F \subset R_i$ and $F \subset Z$ separately. Suppose first that $F \subset R_i$ for some $i$, and let $G_1, \ldots, G_k$ be the irreducible components of $E$ adjacent to $F$. Since the dual graph of $R_i$ is a tree meeting $Z$ in a single point, we may order the $\{G_i\}$ so that
\begin{align*}
l(G_1,Z)&=l(F,Z)-1, \\
l(G_i,Z)&=l(F,Z)+1, &2 \leq i \leq k.\\
\intertext{Since $F$ is rational and the total space $\C^s$ is regular, we have}
\deg\omega_{\C^{s}/\Delta}|_F &=-k-2\\
F.F &=-k\\
G_i.F &=1, &1 \leq i \leq k.
\end{align*}
Now, since $F, G_1, \ldots, G_k$ are the only components of $D$ meeting $F$, we have
\begin{align*}
\deg \omega_{\C^s/\Delta}(D)|_{F}=&\deg \omega_{\C^{s}/\Delta}|_{F}+(l+1-l(F,Z))(F.F)\\
&+(l+2-l(F,Z))(G_1.F)+(l-l(F,Z))((G_2+ \ldots+G_k).F)\\
=&(-k-2)+(l+1-l(F,Z))(-k)+(l+2-l(F,Z))+(l-l(F,Z))(k-1)\\
=& 0.
\end{align*}
Since $F$ is rational, this implies $\omega_{\C^s/\Delta}(D)|_{F} \simeq \O_{F}$. 

It remains to show that $\omega_{\C^s/\Delta}(D)|_{Z} \simeq \O_{Z}$. First, note that $\omega_{Z}=\O_{Z}$ (Recall that $Z$ is irreducible of arithmetic genus one, or a ring of $\P^1$'s). If $G_1, \ldots, G_k$ are the components of $E$ adjacent to $Z$, then $l(G_i,Z)=1$, so we have
\begin{align*}
\omega_{\C^s/\Delta}(D)|_{Z} &\simeq \omega_{\C^s/\Delta} \left( (l+1)Z + lG_1+ \ldots + lG_k \right)| _{Z}\\
& \simeq \omega_{Z} \otimes \O_{\C^s}(lZ+lG_1+\ldots+lG_k)|_{Z}\\
& \simeq \O_{Z}.
\end{align*}
\end{proof}

\begin{proof}[$(E,p_1, \ldots, p_m)$ a semistable tail $\implies (E,p_1, \ldots, p_m)$ balanced]
Suppose $(E,p_1, \ldots, p_m)$ is a semistable tail of the elliptic $m$-fold point. Then we have a smoothing $\C/\Delta$, a semistable limit $\C^{s}/\Delta$, and a birational morphism $\phi: \C^{s} \rightarrow \C$ with exceptional curve $E$. (Replacing $\C/\Delta$ by $\C \times_{\Delta} \Delta'/\Delta'$, we may assume that the semistable limit is defined over the same base as the smoothing.) Set
$$
\tilde{C}:= \overline{C^{s} \backslash E},
$$
and note that the restriction of $\phi$ to $\tilde{C}$ is precisely the normalization of $C$.

Since $\C \rightarrow \Delta$ and $\C^{s} \rightarrow \Delta$ are Gorenstein morphisms, they are equipped with relative dualizing sheaves and we may consider the \emph{discrepancy} of $\phi$, i.e. we have
$$
\phi^*\omega_{\C^s/\Delta}=\omega_{\C/\Delta}(D),
$$
where $D$ is a Cartier divisor supported on $E$. We may write
$$
D=\sum_{F \subset E}d(F)F,
$$
and we claim that the coefficients $d(F)$ must satisfy the following conditions.
\begin{itemize}
\item[$(A)$] If $F$ meets $\tilde{C}$, then $d(F)=1$.
\item[$(B)$] If $F,G$ are adjacent and $l(F,Z)=l(G,Z)-1$, then $d(F)=d(G)+1$.
\end{itemize}
Condition $(A)$ is easy to see: We have
$$\omega_{\C^{s}/\Delta}(D)|_{\tilde{C}} \simeq (\phi^*\omega_{\C/\Delta})|_{\tilde{C}} \simeq \phi|_{\tilde{C}}^*\, (\omega_{\C/\Delta}|_{C}) \simeq \phi|_{\tilde{C}}^*\,\omega_{C}.$$

Furthermore, since $\phi|_{\tilde{C}}$ is just the normalization of $C$, Proposition \ref{P:dualizingsheaf} implies that 
$$
\phi|_{\tilde{C}}^*\,\omega_{C} \simeq \omega_{\tilde{C}}(2p_1+ \ldots +2p_m).
$$
Putting these two equations together, we  get $\omega_{\C^{s}/\Delta}(D)|_{\tilde{C}} \simeq \omega_{\tilde{C}}(2p_1+ \ldots +2p_m).$ Since $\omega_{\C^{s}/\Delta}|_{\tilde{C}} \simeq \omega_{\tilde{C}}(p_1+ \ldots+p_m),$ $D$ must contain each component that meets $\tilde{C}$ with multiplicity one. This proves $(A)$.

Condition $(B)$ comes from the observation that
\begin{equation*}\label{trivial}
\omega_{\C^{s}/\Delta}(D)|_G \simeq \O_{G}
\end{equation*}
for each irreducible component $G \subset E$, since $E$ is contracted by $\phi$. Indeed, suppose condition $(B)$ fails for a pair of adjacent components $F,G$ with $l(F,Z)=l(G,Z)-1$. Let $H_1, \ldots, H_k$ be the remaining components of $E$ adjacent to $G$ and note that 
\begin{align*}
&l(H_i,Z)=l(G,Z)+1, &1 \leq i \leq  k.\\
\intertext{By choosing the pair $F,G$ with $l(F,Z)$ maximal, we may assume $(B)$ holds for each of the pairs $G,H_i$. Thus,}
&d(H_i,Z)=d(G,Z)-1, &1 \leq i \leq k.
\end{align*}
Since $\omega_{\C^s/\Delta}(D)|_{G} \simeq \O_{G},$ we obtain
\begin{align*}
0&=\deg \omega_{\C^s/\Delta}(D)|_{G}\\
&=\deg\omega_{\C^s/\Delta}|_G+d(F)(F.G)+d(G)(G.G)+(d(G)-1)(H_1.F+ \ldots +H_k.F)\\
&=-2+(k+1)+d(F)+d(G)(-k-1)+(d(G)-1)k,\\
&=d(F)-d(G)-1,
\end{align*}
which gives $d(F)=d(G)+1$ as desired.

Now we will show that conditions $(A)$ and $(B)$ imply that $(E,p_1,\ldots,p_m)$ is balanced. Suppose first that $Z$ is irreducible. Pick a point $p_i \in E \cap \tilde{C}$, and consider a minimum-length path from the irreducible component containing $p_i$ to $Z$. Then $l(-,Z)$ decreases by one as we move along each consecutive component, so conditions $(A)$ and $(B)$ imply that
\begin{align*}
d(Z)=l(p_i,Z).
\end{align*}
Since this holds for each point $p_i \in  E \cap \overline{C^s \backslash E}$, we have
$$
d(Z)=l(p_1,Z)=\ldots=l(p_m,Z),
$$
so $(E,p_1, \ldots, p_m)$ is balanced.

If $Z$ is a ring of $\P^1$'s, and $Z_i \subset Z$ is any irreducible component, then the same argument shows that
\begin{align*}
d(Z_i)=l(p_j,Z),
\end{align*}
for any point $p_j \in E \cap \tilde{C}$ which lies on a connected component of $\overline{C \backslash Z}$ meeting $Z_i.$ Since every connected component of $\overline{C \backslash Z}$ meets some irreducible component of $Z$, $(E,p_1, \ldots, p_k)$ will be balanced if we can show that $d(Z_i)=d(Z_j)$ for all irreducible components $Z_i,Z_j \subset Z$. Since $Z$ is a ring, it suffices to show that for each triple of consecutive components $Z_1,Z_2,Z_3$, we have
$$
2d(Z_2) = d(Z_1)+d(Z_3).
$$
To see this, let $H_1, \ldots, H_k$ be the components of $R$ adjacent to $Z_2$. By condition $(B)$ we have $d(H_i)=d(Z_2)-1$ for each $H_i$. Using  $\omega_{\C/\Delta}(D)|_{Z_2} \simeq \O_{Z_2}$, we obtain
\begin{align*}
0&=\deg \omega_{\C/\Delta}(D)|_{Z_2}\\
&=\deg \omega_{\C/\Delta}|_{Z_2}+d(Z_1)(Z_1.Z_2)+d(Z_3)(Z_3.Z_2)+d(Z_2)(Z_2.Z_2)+(d(Z_2)-1)(H_1.Z_2+ \ldots +H_k.Z_2)\\
&=k+d(Z_1)+d(Z_3)+d(Z_2)(-k-2)+(d(Z_2)-1)k,\\
&=d(Z_1)+d(Z_3)-2d(Z_2)
\end{align*}
which gives $2d(Z_2)=d(Z_1)+d(Z_3)$ as desired.
\end{proof}

\section{Construction of $\SM_{1,\A}(m)$} 
In this section, we turn from local considerations concerning the elliptic $m$-fold point to global considerations of moduli.

\subsection{Fundamental decomposition of a genus one curve}\label{S:fundecomp}
As indicated in the introduction, the reason that we can formulate a separated moduli problem for pointed curves of genus one, but not for higher genus, is the following elementary fact about the topology of a curve of arithmetic genus one.
\begin{lemma}[Fundamental Decomposition]\label{L:decomp} Let $C$ be a Gorenstein curve of arithmetic genus one. Then $C$ contains a unique subcurve $Z \subset C$ satisfying
\begin{itemize}
\item[(1)] $Z$ is connected,
\item[(2)] $Z$ has arithmetic genus one,
\item[(3)] $Z$ has no disconnecting nodes.
\end{itemize}
We call $Z$ the \emph{minimal elliptic subcurve of $C$}. We write
$$
C=Z \cup R_1 \cup \ldots \cup R_k,
$$
where $R_1, \ldots, R_k$ are the connected components of $\overline{C \backslash Z}$, and call this \emph{the fundamental decomposition of $C$}. Each $R_i$ is a nodal curve of arithmetic genus zero meeting $Z$ in a single point, and $Z \cap R_i$ is a node of $C$.
\end{lemma}
\begin{proof}
First, we show that $Z \subset C$ exists. If $C$ itself has no disconnecting nodes, take $Z=C$. If $C$ has a disconnecting node $p$, then the normalization of $Z$ at $p$ will comprise two connected components, one of which has arithmetic genus one. Proceed by induction on the number of disconnecting nodes.

Next we show that the connected components of $\overline{C \backslash Z}$ each have arithmetic genus zero, and meet $Z$ in a single point, which is a simple node of $C$. If $R_1, \ldots, R_k$ are the connected components of $\overline{C \backslash Z}$ and $p_1, \ldots, p_l$ the points of intersection $Z \cap (R_1 \cup \ldots \cup R_k)$, we have
$$
1=p_a(C)=p_a(Z)+\sum_{i=1}^{k}p_{a}(R_i)+\sum_{i=1}^{l}\delta(p_i)+1-k.
$$
Since $p_a(R_i) \geq 0$, $\delta(p_i) \geq 1$, and $l \geq k$, we see that equality holds iff $p_a(R_i) = 0$, $\delta(p_i)=1$, and $l=k$. Since $R_i$ is Gorenstein of arithmetic genus zero, Proposition \ref{P:genus0} implies that $R_i$ is nodal. Since $p_i$ is a Gorenstein curve singularity with $\delta(p_i)=1$ and at least two branches, it must have exactly two branches. Then Corollary \ref{C:gorenstein} implies that $p_i$ is a node. Finally, the fact that $l=k$ says precisely that each connected component $R_i$ meets $Z$ in a single point.

It remains to show that $Z$ is unique. By symmetry, it is enough to show that if $Z'$ satisfies (1)-(3) then $Z' \subset Z$. If this fails then $Z' \cap R_i \neq \emptyset$ for some $i$. Since $p_a(Z')=1$, $Z'$ cannot be contained in $R_i$, so $Z'$ meets $Z$. But then, since $Z'$ is connected, $Z'$ contains the disconnecting node $R_i \cap Z$, a contradiction.
\end{proof}
\begin{corollary}\label{C:minimality}
Let $C$ be a Gorenstein curve of arithmetic genus one with minimal elliptic subcurve $Z$. If $E \subset C$ is any connected arithmetic genus one subcurve of $C$, then $Z \subset E$.
\end{corollary}
\begin{proof}
The minimal elliptic subcurve of $E$ is necessarily the minimal elliptic subcurve of $C$, namely $Z$. Thus, $Z \subset E$.
\end{proof}

The following lemma gives an exact characterization of the `minimal elliptic subcurves' appearing in Lemma \ref{L:decomp}.
\begin{lemma}\label{L:MinEllipticSub}
Suppose $Z$ is Gorenstein of arithmetic genus one and has no disconnecting nodes. Then $Z$ is one of the following:
\begin{itemize}
\item[(1)] A smooth elliptic curve,
\item[(2)] An irreducible rational nodal curve,
\item[(3)] A ring of $\P^{1}$'s, or
\item[(4)] $Z$ has an elliptic $m$-fold point $p$ and the normalization of $Z$ at $p$ consists of $m$  distinct, smooth rational curves.
\end{itemize}
Furthermore, in all four cases, $\omega_{Z} \simeq \O_{Z}$.
\end{lemma}
\begin{proof}
First, suppose $Z$ has a non-nodal singular point $p$. Then by Corollary \ref{C:gorenstein}, $p$ is an elliptic $m$-fold point for some integer $m$ and the normalization of $Z$ at $p$ consists of $m$ distinct connected nodal curves of arithmetic genus zero. But a (nodal) curve of arithmetic genus zero with no disconnecting nodes must be smooth, so (4) holds.

Next, suppose $Z$ has only nodes. If $Z$ is smooth, we are in case (1) so assume there exists a node $p$. Then $\tilde{Z}$, the normalization of $Z$ at $p$, is connected, nodal, and has arithmetic genus zero. If $\tilde{Z}$ is smooth, we are in case (2). Otherwise, $\tilde{Z}$ is a tree of $\P^{1}$s. To see that we are in case (3), it is sufficient to show that $\tilde{Z}$ is actually a chain of $\P^{1}$'s, i.e. that the only irreducible components $F \subset \tilde{Z}$ with the property that
$$
| F \cap \overline{\tilde{Z} \backslash F}|=1
$$
are the two irreducible components lying over $p$. But if $F \subset \tilde{Z}$ satisfies $| F \cap \overline{\tilde{Z} \backslash F}|=1$ and $F$ does not lie over $p$, then $F \cap \overline{\tilde{Z} \backslash F}$ is a disconnecting node of $Z$, a contradiction.

In cases (1)-(3), the isomorphism $\omega_{Z} \simeq \O_{Z}$ is clear. In case (4), we will write down a nowhere vanishing global section of $\omega_{Z}$. Let $\tilde{Z}_{1}, \ldots, \tilde{Z}_{m}$ be the connected components of $\tilde{Z}$ and $p_i \in \tilde{Z}_{i}$ the point lying over $p$. We may choose local coordinates $t_i$ at $p_i$ so that the map $\tilde{Z} \rightarrow Z$ is given by the expression $(\dag)$ in Definition \ref{D:mpoint}. Since each $\tilde{Z}_{i} \simeq \P^{1}$, the rational differential
$$
\frac{dt_1}{t_1^2} \in H^0(\tilde{Z}_i, \omega_{\tilde{Z}_i}(2p_i))
$$
gives a global section of $\omega_{\tilde{Z}_i}(2p_i)$, regular and non-vanishing away from $p_i$. The proof of Proposition \ref{P:dualizingsheaf} shows that
$$\frac{dt_1}{t_1^2}+ \ldots \frac{dt_{m-1}}{t_{m-1}^2}-\frac{dt_{m}}{t_m^2} \in H^0(\omega_{\tilde{Z}}(2p_1+ \ldots 2p_m))$$
descends to a section of $\omega_{Z}$ which generates $\omega_{Z}$ locally around $p$. Thus, it generates $\omega_{Z}$ globally.

\end{proof}

In order to define and work with the moduli problem of $m$-stable curves, it is useful to have the following terminology.

\begin{definition}[Level]\label{D:level}
Let $(C,p_1, \ldots, p_n)$ be an $n$-pointed curve of arithmetic genus one, let $Z \subset C$ be the minimal elliptic subcurve of $C$, and let $\Sigma \subset C$ denote the support of the divisor $\sum_{i}p_i$. The \emph{level} of $(C,p_1, \ldots, p_n)$ is defined to be the integer
$$|Z \cap \overline{C \backslash Z}| + |Z \cap \Sigma |.$$
\end{definition}

\begin{lemma}\label{L:level}
Suppose $(C,p_1, \ldots, p_n)$ is an $n$-pointed curve of arithmetic genus one and suppose every smooth rational component of $C$ has at least two distinguished points. Let $Z \subset C$ be the minimal elliptic subcurve, and $\Sigma \subset C$ the support of the divisor $\sum_{i}p_i$. If $E$ is any connected subcurve of arithmetic genus one, then
$$|E \cap \overline{C \backslash E}|+|E \cap \Sigma | \geq |Z \cap \overline{C \backslash Z}| + | Z \cap \Sigma|.$$
\end{lemma}
\begin{proof}
Let $C=Z \cup R_1 \cup \ldots \cup R_k$ be the fundamental decomposition of $C$, and order the $R_{i}$ so that $E$ contains $R_1, \ldots, R_j$, but does not contain $R_{j+1}, \ldots, R_{k}$. The assumption that each smooth rational component has at least two distinguished points implies that each of $R_{1}, \ldots, R_{j}$ contains at least one marked point so
$$
|E \cap \Sigma | \geq |Z \cap \Sigma|+j.
$$
On the other hand, since $E$ does not contain $R_{j+1}, \ldots, R_{k}$, we must have
$$|E \cap \overline{C \backslash E}| \geq |Z \cap \overline{C \backslash Z}|-j.$$
Thus, 
$$|E \cap \overline{C \backslash E}|+|E \cap \Sigma | \geq |Z \cap \overline{C \backslash Z}| + |Z \cap \Sigma|.$$
\end{proof}
\begin{corollary}\label{C:level}
Let $(C,p_1, \ldots, p_n)$ be an $n$-pointed curve of arithmetic genus one, and suppose that every smooth rational component has at least two distinguished points. Then $(C,p_1, \ldots, p_n)$ has level $>m$ iff
$$|E \cap \overline{C \backslash E}| + |E \cap \Sigma|>m$$
for every connected arithmetic genus one subcurve $E \subset C$.
\end{corollary}

\subsection{Definition of the moduli problem}\label{S:mstable}
We are ready to define the moduli problem of $(m, \A)$-stable curves.
\begin{definition}[$(m, \A)$-stability]\label{D:mstable}
Fix positive integers $m<n$, and a vector of rational weights $\A=(a_1, \ldots, a_n) \in (0,1]^{n}$. Let $(C,p_1, \ldots, p_n)$ be an $n$-pointed curve of arithmetic genus one, and let $\Sigma \subset C$ denote the support of the divisor $\sum_{i}p_i$. We say that $C$ is \emph{$(m, \A)$-stable} if
\begin{itemize}
\item[(1)] The singularities of $C$ are nodes or elliptic $l$-fold points, $l \leq m$.
\item[(2)] The level of $(C,p_1, \ldots, p_n)$ is $> m$. Equivalently, by Corollary \ref{C:level},
$$|E \cap \overline{C \backslash E}| + |E \cap \Sigma |>m$$
for every connected arithmetic genus one subcurve $E \subset C$.
\item[(3)] $H^0(C,\Omega_C^{\vee}(-\Sigma))=0.$ Equivalently, by Corollary \ref{C:automorphisms}, 
\begin{itemize}
\item[(a)] If $C$ is nodal, then every rational component of $\tilde{C}$ has at least three distinguished points.
\item[(b)] If $C$ has a (unique) elliptic $m$-fold point $p$, and $\tilde{B}_1, \ldots, \tilde{B}_{m}$ denote the components of the normalization whose images contain $p$, then
\begin{itemize}
\item[(b1)] $\tilde{B}_1, \ldots, \tilde{B}_m$ each have $\geq 2$ distinguished points.
\item[(b2)] At least one of $\tilde{B}_1, \ldots, \tilde{B}_m$ has $\geq 3$ distinguished points.
\item[(b3)] Every other component of $\tilde{C}$ has $\geq 3$ distinguished points.
\end{itemize}
\end{itemize}
\item[(4)] If $p_{i_1}=\ldots=p_{i_k} \in C$ coincide, then $\sum_{j=1}^{k}a_{i_j} \leq 1.$ 
\item[(5)] $\omega_{C}(\Sigma_ia_ip_i)$ is an ample $\Q$-divisor.
\end{itemize}
\end{definition}
\begin{remark}
When $\A=(1, \ldots, 1)$, then we say simply that $(C, p_1, \ldots, p_n)$ is $m$-stable. In this case, condition (4) merely asserts that the marked points are distinct, and condition (5) follows automatically from condition (3). Indeed, conditions (b1) and (b3) above, combined with Proposition \ref{P:dualizingsheaf}, imply that $\omega_{C}(\Sigma_{i}p_i)$ has positive degree on every component of $C$.
\end{remark}

The definition of an $(m, \A)$-stable curve extends to a moduli functor in the usual way.
If $S$ is an arbitrary scheme over $\Spec \mathbb{Z}[1/6]$, an \emph{$(m, \A)$-stable curve over $S$} consists of a morphism of schemes $\pi:X \rightarrow S$, together with $n$ sections $\sigma_1, \ldots, \sigma_n$, such that
\begin{itemize}
\item[(1)] $\pi$ is flat, projective, and locally of finite-presentation,
\item[(2)] The images of $\sigma_1, \ldots \sigma_n$ lie in the smooth locus of $\pi$,
\item[(3)] For any point $s \in S$, the geometric fiber $(X_{\overline{s}}, \sigma_1(\overline{s}),  \ldots ,  \sigma_n(\overline{s}))$ is an $(m, \A)$-stable curve over $\overline{k(s)}$.
\end{itemize}
A morphism of $(m, \A)$-stable curves, from $(X/S, \sigma_1, \ldots, \sigma_n)$ to $(Y/T, \tau_1, \ldots, \tau_n)$, is a commutative diagram
\[
\xymatrix{
X \ar[d] \ar[r]^{\phi} & Y \ar[d] \\
S  \ar@/^1pc/[u]^{\{\sigma_i\}}  \ar[r] & T \ar@/_1pc/[u]_{\{\tau_i\}} 
}
\] 
such that the induced map $X \rightarrow Y \times_{T} S$ is an isomorphism, and $\phi \circ\sigma_i = \tau_i$ for $i=1, \ldots, m$. The assignment
$$
(X/S, \sigma_1, \ldots, \sigma_n) \rightarrow \omega_{X/S}(\Sigma_{i}a_i\sigma_i) \in \Pic(X) \otimes \Q
$$
gives a canonical $\Q$-polarization for our moduli problem, so
$(m, \A)$-stable curves (and morphisms of $(m, \A)$-stable curves) satisfy \'{e}tale descent, i.e. they form a stack $\SM_{1,\A}(m)$. The main theorem of this paper is
\begin{theorem}\label{T:MainResult1}
$\SM_{1,\A}(m)$ is a proper irreducible Deligne-Mumford stack over $\Spec \mathbb{Z}[1/6].$
\end{theorem}
We will prove that the moduli problem of $(m, \A)$-stable curves is bounded and deformation-open in Lemmas \ref{L:bounded} and \ref{L:defopen}, and we verify the valuative criterion in Sections \ref{S:ValuativeCriterion} and \ref{S:mAstability}. Everything else follows by standard arguments which we outline below.
\begin{proof}
To say that $\SM_{1,\A}(m)$ is an algebraic stack of finite-type over $\Spec \mathbb{Z}[1/6]$ means  \cite{LMB}:
\begin{itemize}
\item[(1)] The diagonal $\Delta: \SM_{1,\A}(m) \rightarrow \SM_{1,\A}(m) \times \SM_{1,\A}(m)$ is representable, quasicompact, and of finite-type.
\item[(2)] There exists an irreducible scheme $U$, of finite-type over $\Spec \mathbb{Z}[1/6]$, with a smooth, surjective morphism $U \rightarrow \SM_{1,\A}(m).$
\end{itemize}
Since $m$-stable curves are canonically polarized, the Isom-functor for any pair of $m$-stable curves over $S$ is representable by a quasiprojective scheme over $S$, which gives (1).

For (2), fix an integer $N>n+\max\{2m,4\}$ as in the boundedness statement of Lemma \ref{L:bounded}, and assume that $N$ is sufficiently divisible so that each $Na_i \in \mathbb{Z}$. Set
\begin{align*}
d&=N(\Sigma_{i}a_i),\\
r&=N(\Sigma_{i}a_i)-1.
\end{align*}
If $(C,p_1, \ldots, p_n)$ is any $m$-stable curve, Riemann-Roch implies
\begin{align*}
d&=\deg \omega_C(\Sigma_{i}p_i)^{\otimes N},\\
r&=\dim H^0(C, \omega_C(\Sigma_{i}p_i)^{\otimes N})-1,
\end{align*}
so Lemma \ref{L:bounded} implies that every $N$-canonically polarized $m$-stable curve appears in the Hilbert scheme of curves of degree $d$ and arithmetic genus one in $\P^r:=\P^{r}_{\mathbb{Z}[1/6]}$. Let $\H$ denote this Hilbert scheme and consider the locally-closed subscheme
\begin{align*}
Z&=\{(C,p_1, \ldots, p_n) \subset \H \times (\P^r)^n \,| \,\, p_1, \ldots, p_n \text{ are smooth points of } C \}.
\intertext{By Lemma \ref{L:defopen}, there exists an open subscheme of $Z$ defined by}
V&=\{(C,p_1, \ldots, p_n) \subset \H \times (\P^r)^n \,| \,\, (C,p_1, \ldots, p_n) \text{ is $m$-stable} \}.
\end{align*}

Using the representability of the Picard scheme \cite[Ch. 5]{GIT}, there exists a locally-closed subscheme $U \subset V$, such that
$$
U=\{ (C,p_1, \ldots, p_n) \subset V \, |\,\, \omega_C(\Sigma_i a_ip_i)^{\otimes N} \simeq \O_C(1) \}.
$$
Now the classifying map $U \rightarrow \SM_{1,n}(m)$ is smooth and surjective.

To show that $\SM_{1,\A}(m)$ is Deligne-Mumford over $\Spec \mathbb{Z}[1/6]$, it suffices to show that if $k$ is an algebraically closed field and $\characteristic k \neq 2,3$, then the group scheme $\Aut_k(C,p_1, \ldots, p_n)$ is unramified over $k$. There is a natural identification of $k[\epsilon]/(\epsilon^2)$-points of $\Aut_k(C,p_1, \ldots, p_n)$ with global sections of $\Omega^{\vee}_{C}(-\Sigma)$ \cite[3.3]{Hassett4}, so this follows from condition (3) in the definition of an $(m, \A)$-stable curve.

Finally, to show that $\SM_{1,\A}(m)$ is irreducible, it is sufficient to show that $\mathcal{M}_{1,n} \subset \SM_{1,\A}(m)$ is dense, i.e. that every $m$-stable curve is smoothable. Since a curve is smoothable iff each of its singularities is smoothable \cite[II.6.3]{Kollar2}, and the only singularities on an $m$-stable curve are elliptic $l$-fold points and nodes, it suffices to see that the elliptic $l$-fold point is smoothable. This is an old result going back to Pinkham \cite{Pinkham}, but we may also note that we have constructed explicit smoothings of the elliptic $l$-fold point in our proof of Proposition \ref{P:semistablelimits}.
\end{proof}

\begin{lemma}[Boundedness]\label{L:bounded}
If $(C,p_1, \ldots, p_n)$ is any $(m, \A)$-stable curve, then the line-bundle 
$$L^N:=\omega_{C}(\Sigma_ia_ip_i)^{\otimes N}$$
is very ample on $C$ for any $N>n+\max\{2m,4\}$ and sufficiently divisible.
\end{lemma}
\begin{proof}
Throughout this argument, we will assume that $N$ is chosen sufficiently divisible so that $\omega_{C}(\Sigma_ia_ip_i)^{\otimes N}$ is integral. With this caveat, it is enough to show that $N>n+\max\{2m,4\}$ implies
\begin{itemize}
\item[(1)] $H^1(C, L^N \otimes I_{p})=0$ for any point $p \in C$, 
\item[(2)] $H^1(C, L^N \otimes I_{p}I_{q})=0$ for any pair of points $p,q \in C$.
\end{itemize}
Condition (1) says that the complete linear series $H^0(C,L^N)$ is basepoint-free, while condition (2) says that it separates points ($p \neq q$) and tangent vectors ($p=q$). Clearly $(2) \implies (1)$. Using Serre duality, it is enough to show that
$$
H^0(C, \omega_{C} \otimes L^{-N} \otimes (I_pI_q)^{\vee})=0.
$$
Let $\pi: \tilde{C} \rightarrow C$ be the normalization of $C$ at $p$ and $q$, with $p_1, \ldots p_k$ the points of $\tilde{C}$ lying above $p$, and $q_1, \ldots, q_l$ the points lying above $q$. Define
$$
D:=\sum_{i=1}^{m}2p_i+\sum_{j=1}^{l}2q_i
$$
as a Cartier divisor on $\tilde{C}$, and note that $\deg D \leq \max\{4,2m\}$ (since any singular point of $C$ has at most $\max\{2,m\}$ branches). By Lemma \ref{L:mpoint},
\begin{align*}
 \pi_{*}\O_{\tilde{C}}(-D) \subset I_{p}I_{q},
\end{align*}
and the quotient is torsion, supported at $\{p\} \cup \{q\}$. Thus, we obtain injections
$$
\Hom(I_pI_q, \O_{C}) \hookrightarrow \Hom(\pi_{*}\O_{\tilde{C}}(-D) , \O_{C}) \hookrightarrow \pi_*\Hom(\O_{\tilde{C}}(-D) , \O_{\tilde{C}}). 
$$
Tensoring by $\omega_{C} \otimes L^{-N}$, we obtain
$$
(I_pI_q)^{\vee} \otimes (\omega_{C} \otimes L^{-N}) \hookrightarrow \pi_*\O_{\tilde{C}}(D) \otimes (\omega_{C} \otimes L^{-N}),
$$
so that
$$
H^0(\tilde{C}, \O_{\tilde{C}}(D) \otimes \pi^*(\omega_{C} \otimes L^{-N}))=0 \implies H^0(C, \omega_{C} \otimes L^{-N} \otimes (I_pI_q)^{\vee})=0.
$$
We claim that $N>n+\max\{4,2m\}$ forces the line-bundle $\O_{\tilde{C}}(D) \otimes \pi^*(\omega_{C} \otimes L^{-N})$ to have negative degree on each component (and hence no sections).
Since $\pi^*L$ has degree at least one on every component of $\tilde{C},$ and $\deg D \leq \max\{4,2m\}$, it is enough to show that $\pi^*\omega_{C}$ has degree at most $n$ on any irreducible component $ F \subset \tilde{C}.$ To see this, simply observe
$$
\deg_{F}\pi^*\omega_{C} \leq \deg_{F}\pi^*\omega_{C}(\Sigma_i p_i) \leq n,
$$
where the last inequality follows from the fact that $\pi^*\omega_{C}(\Sigma_i p_i)$ has total degree $n$ and non-negative degree on each component.
\end{proof}

\begin{lemma}[Deformation-Openness]\label{L:defopen} Let $S$ be a noetherian scheme and let $(\phi: \C \rightarrow S, \sigma_1, \ldots, \sigma_n)$ be a flat, projective morphism of relative dimension one with $n$ sections $\sigma_1, \ldots, \sigma_n$. The set
$$
T=\{s \in S | (C_{\overline{s}}, \sigma_1(\overline{s}),  \ldots ,  \sigma_n(\overline{s})) \text{ is $m$-stable}\}
$$
is Zariski-open in $S$.
\end{lemma}
\begin{proof}
We may assume that the fibers of $\phi$ are reduced, connected, and of arithmetic genus one, since these are all open conditions \cite[12.2]{EGAIV}. We may also assume that the geometric fibers are Gorenstein (the locus in $S$ over which the geometric fibers are Gorenstein is the same as the locus over which the relative dualizing sheaf $\omega_{\C/S}$ is invertible, hence open). Finally, the conditions that the sections lie in the smooth locus of $\phi$, that $\sigma_{i_1}, \ldots, \sigma_{i_k}$ collide only if $\sum_{j=1}^{k}a_{i_j} \leq 1$, and that $\omega_{\C/S}(\Sigma_ia_i)$ is relatively ample are obviously open. It only remains to check conditions (1)-(3) of Definition \ref{D:mstable}.

For condition (1), suppose that $s \in S$ is a geometric point and that the fiber $C_{s}$ has an elliptic $m$-fold point $p$. We must show there exists an open neighborhood of $s$ over which the fibers of $\C$ have only elliptic $l$-fold points, $l \leq m$, and nodes. Suppose first that $m \geq 3$. Since the dimension of the Zariski tangent space of the ellliptic $m$-fold point is $m$ when $m \geq 3$, we have 
$$\dim_{k(x)}m_{x}/m_{x}^2 \leq m \text{ for every $x \in C_{s}$,}$$
where $m_{x}$ refers to the maximal ideal of $x$ in the local ring of the fiber. Thus, there is an open neighborhood of the fiber $C \subset V \subset X$ such that
$$\dim_{k(x)}m_{x}/m_{x}^2 \leq m \text{ for every $x \in V.$}$$
Since $\pi$ is proper, we may take $V$ to be of the form $\pi^{-1}(U)$ for some open set $U \subset S$. Now, for any $s \in U$, the fiber is a Gorenstein curve of arithmetic genus one whose Zariski tangent space dimension is everywhere $\leq m$. By Proposition \ref{P:genus1}, the only singularities appearing on fibers over $U$ are elliptic $l$-fold points, $l \leq m$, and nodes.

It remains to consider the case $m=2$ or $m=3$, i.e $p \in C_{s}$ is cusp or node. For this, we need a bit of deformation theory \cite{Sch}. Recall that the cusp and tacnode, being local complete intersections, admit versal deformations given by:
\begin{align*}
\Spec A[x,y,a,b]/(y^2=x^3+ax+b) &\rightarrow \Spec A[a,b],\\
\Spec A[x,y,a,b,c]/(y^2=x^4+ax^2+bx+c) &\rightarrow \Spec A[a,b,c],
\end{align*}
where $A=k(s)$ if $\characteristic k(s)=0$ or the unique complete local ring with residue field $k(s)$ and maximal ideal $pA$ if $\characteristic k(s)=p$. If $p \in C$ is a cusp (resp. tacnode), there is an etale neighborhood $(U,0) \rightarrow (S,s)$, and a map
$$
U \rightarrow \Spec A[a,b] \text{ ($\Spec A[a,b,c]$), }
$$
such that, etale-locally around $p \in C$, $\C \times_{S} U$ is pulled back from the versal family. Since the only singularities appearing in fibers of the versal deformation of the cusp (resp. tacnode) are nodes (resp. nodes and cusps), we are done.

For conditon (2), we must show that the locus in $S$ over which the fibers have level $>m$ is open in $S$. Since $S$ is noetherian, it suffices to show that this locus is constructible and stable under generalization. It is clearly constructible, since we may stratify $T$ into locally-closed subsets corresponding to the topological type of the fiber and observe that the level of a fiber depends only on the topological type. To see that it is stable under generalization, we may assume that $S$ is the spectrum of a discrete valuation ring with closed point $0 \in S$ and generic point $\eta \in S$. We must show that if $(C_{0}, \sigma_1(0), \ldots, \sigma_n(0))$ has level $>m$, then so does $(C_{\overline{\eta}},\sigma_1(\overline{\eta}), \ldots, \sigma_n(\overline{\eta}))$.

Let $E_{\overline{\eta}}$ be a connected arithmetic genus one subcurve of the geometric generic fiber $C_{\overline{\eta}}$. The limit of $E_{\overline{\eta}}$ in the special fiber is a connected arithmetic genus one subcurve $E_0 \subset C_0$ satisfying
\begin{align*}
| E_{\overline{\eta}} \cap  \overline{C_{\overline{\eta}} \backslash E_{\overline{\eta}}}  | &=|  E_0 \cap \overline{C_0 \backslash E_0} |, \\
 |E_{\overline{\eta}} \cap \Sigma_{\overline{\eta}}|&=|E_0 \cap \Sigma_0 |.
\end{align*}
Since $(C_0,\sigma_1(0), \ldots, \sigma_n(0))$ has level $>m$, we have
\begin{align*}
| E_{\overline{\eta}} \cap  \overline{C_{\overline{\eta}} \backslash E_{\overline{\eta}}}  | +  |E_{\overline{\eta}} \cap \Sigma|=|E \cap \overline{C \backslash E}|+|E \cap \Sigma | > m.
\end{align*}
Thus, $(C_{\overline{\eta}},\sigma_1(\overline{\eta}), \ldots, \sigma_n(\overline{\eta}))$ has level $>m$, as desired.

For condition (3), using the natural identification between $k[\epsilon]/\epsilon^{2}$-points of $\Aut_k(C,\{p_i\})$ and global sections of $\Omega^{\vee}_{C}(-\Sigma)$, it suffices to show that the locus $U \subset S$ over which the group scheme
$
\Aut_{S}(\C, \sigma_1, \ldots, \sigma_n) \rightarrow S
$
is unramified is open in $S$. But this is a general fact about group schemes: Suppose that $\pi:G \rightarrow S$ is any finite-type group scheme over a noetherian base with identity section $e:S \rightarrow G$, and suppose that $\pi$ is unramified over a point $s \in S$. Since the condition of being unramified is open on the domain, there is an open neighborhood $e(s) \in W \subset G$ such that $\pi|_{W}$ is unramified. Setting $U:=e^{-1}(W) \subset S$, we may use translations to cover $\pi^{-1}(U)$ by open sets over which $\pi$ is unramified.

\end{proof}

\subsection{Valuative criterion for $\SM_{1,n}(m)$} \label{S:ValuativeCriterion}
To show that $\SM_{1,\A}(m)$ is proper, it suffices to verify the valuative criterion for discrete valuation rings with algebraically closed residue field, whose generic point maps into the open dense substack $\mathcal{M}_{1,n}$ \cite[Ch. 7]{LMB}. Thus, the required statement is:

\begin{theorem}[Valuative Criterion for Properness of $\SM_{1,\A}(m)$]\label{T:Valuative}
Let $\Delta$ be the spectrum of a discrete valuation ring with algebraically closed residue field, and let $\eta \in \Delta$ be the generic point.
\begin{itemize}
\item[]
\item[(1)] (Existence of $(m, \A)$-stable limits) If $(\C,\sigma_1, \ldots, \sigma_n)|_{\eta}$ is a smooth $n$-pointed curve of arithmetic genus one over $\eta$, there exists a finite base-change $\Delta' \rightarrow \Delta$, and an $(m, \A)$-stable curve $(\C' \rightarrow \Delta', \sigma'_1, \ldots, \sigma'_n)$, such that
$$(\C', \sigma_1', \ldots, \sigma_n' )|_{\eta'} \simeq (\C,\sigma_1, \ldots, \sigma_n)|_{\eta} \times_{\eta} \eta'.$$
\item[(2)] (Uniqueness of $(m, \A)$-stable limits) Suppose that $(\C \rightarrow \Delta,\sigma_1, \ldots, \sigma_n)$  and $(\C' \rightarrow \Delta, \sigma'_1, \ldots, \sigma'_n)$ are $(m, \A)$-stable curves with smooth generic fiber. Then any isomorphism over the generic fiber
$$(\C,\sigma_1, \ldots, \sigma_n)|_{\eta} \simeq (\C',\sigma'_1, \ldots, \sigma'_n)|_{\eta}$$ extends to an isomorphism over $\Delta$:
$$(\C,\sigma_1, \ldots, \sigma_n) \simeq (\C',\sigma'_1, \ldots, \sigma'_n).$$
\end{itemize}
\end{theorem}
In this section, we will prove existence and uniqueness of $m$-stable limits, i.e. we will restrict to the special case $\A=(1, \ldots, 1)$. This will allow us exhibit the main ideas of the proof with a minimum of notational obfuscation. In section \ref{S:mAstability}, we will show that the existence and uniqueness of $(m,\A)$-stable limits can be deduced from the corresponding statement for $m$-stable limits, in the same way that the existence and uniqueness of $\A$-stable limits are deduced from the corresponding statement for Deligne-Mumford stable limits.
\subsubsection{Existence of $m$-stable Limits}
Given a one-parameter family of smooth curves over $\eta$, we construct the $m$-stable limit in three steps: First, we may assume (after a finite base-change) that this family extends to a semistable curve with smooth total space. In step two, we blow-up marked points on the minimal elliptic subcurve of the special fiber, and then contract the strict transform of the minimal elliptic subcurve using Lemma \ref{L:Contraction}. Repeating this process, one eventually reaches a stage where the minimal elliptic subcurve $Z$ satisfies
$$
|Z \cap \overline{C \backslash Z}| + |\{p_i \,|\, p_i \in Z\}| > m.
$$
At this point, we `stabilize,' i.e. blow-down all smooth $\P^{1}$'s which meet the rest of the fiber in two nodes and have no marked points, or meet the rest of the fiber in a single node and have one marked point. The entire process is pictured in figure \ref{F:valuativecriterion}.

\begin{figure}
\scalebox{.55}{\includegraphics{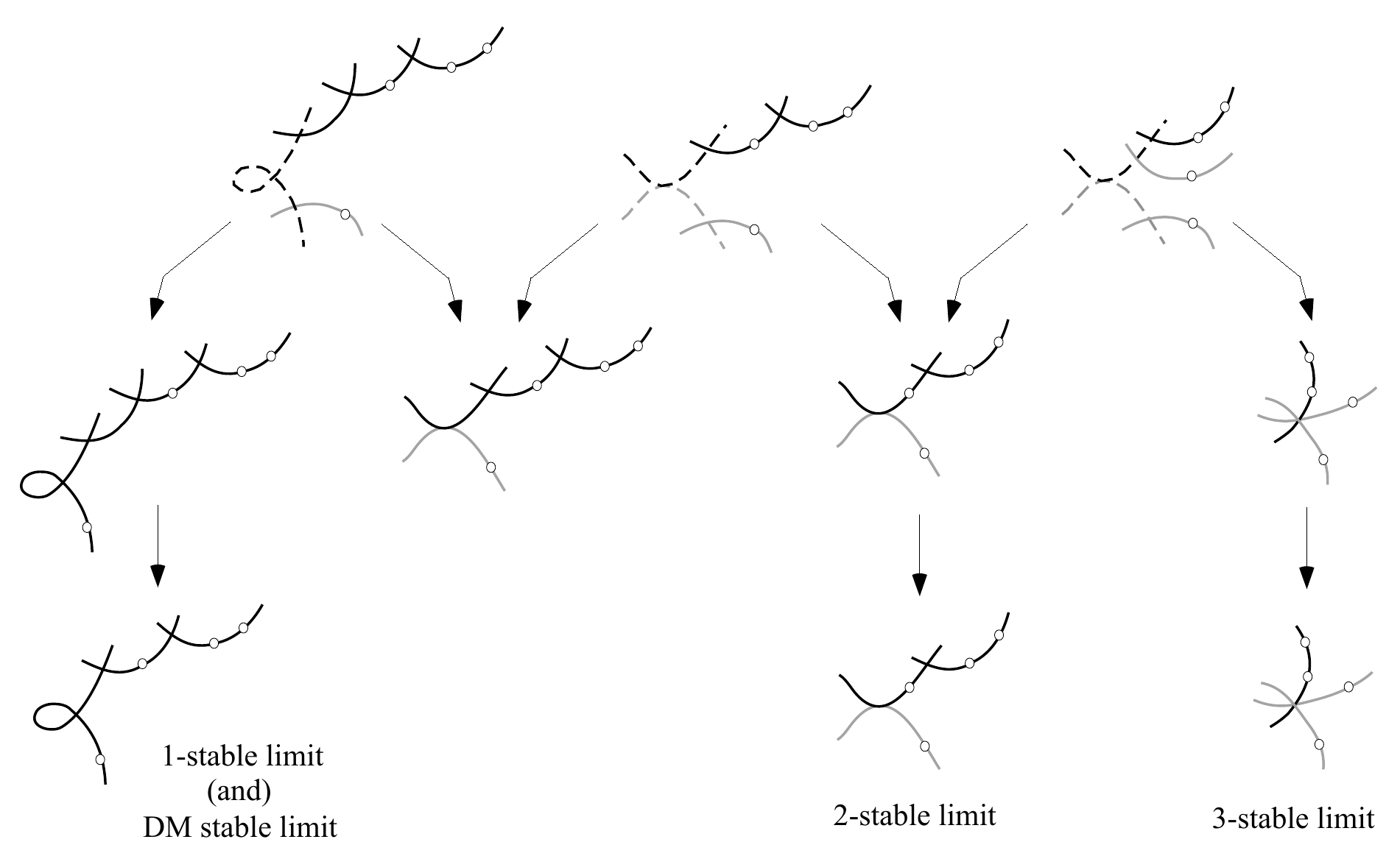}}
\caption{The process of blow-up/contraction/stabilization in order to extract the $m$-stable limit for each $m=1,2,3$. Every irreducible component pictured above is rational. The left-diagonal maps are simple blow-ups along the marked points of the minimal elliptic subcurve, and exceptional divisors of these blow-ups are colored grey. The right-diagonal maps contract the minimal elliptic subcurve of the special fiber, and exceptional components of these contractions are dotted. The vertical maps are stabilization morphisms, blowing down all semistable components of the special fiber.}\label{F:valuativecriterion}
\end{figure}

\begin{Step1} Pass to a semistable limit with smooth total space.
\end{Step1}
By the semistable reduction theorem \cite{DM}, there exists a finite base-change $\Delta' \rightarrow \Delta$, and a semistable curve $(\C^{ss} \rightarrow \Delta',\sigma'_1, \ldots, \sigma'_n)|_{\eta}$ such that
$$
(\C^{ss}, \sigma'_1, \ldots, \sigma'_n )|_{\eta'} \simeq (C,\sigma_1, \ldots, \sigma_n) \times_{\eta} \eta'.
$$
After taking a minimal resolution of singularities, we may assume that the total space of $\C^{ss}$ is regular. For notational simplicity, we will continue to denote our base by $\Delta$, and the given sections  by $\sigma_1, \ldots, \sigma_n$.

\begin{Step2} Alternate between blowing up marked points contained on the minimal elliptic subcurve and contracting the minimal elliptic subcurve.
\end{Step2}
Starting from $\C_{0}:=\C^{ss}$, we construct a sequence $\C_{0}, \C_{1}, \ldots, \C_{t}$ of flat proper families over $\Delta$ satisfying
\begin{itemize}
\item[(i)] The special fiber $C_{i} \subset \C_{i}$ is a Gorenstein curve of arithmetic genus one.
\item[(ii)] The total space $\C_{i}$ is regular at every node of $C_i$.
\item[(iii)] The strict transforms of $\sigma_1, \ldots, \sigma_n$ on $\C_i$ are contained in the smooth locus of $\pi_i$, so we may consider the special fiber as an $n$-pointed curve $(C_i, p_1, \ldots, p_n)$.
\item[(iv)] Every component of $C_{i}$ has at least two distinguished points.
\item[(v)] $C_{i}$ has an elliptic $l_{i-1}$-fold point $p$, where $l_i$ denotes the level of the special fiber $C_i$ (Definition \ref{D:level}).
\item[(vi)] $l_{i} \geq l_{i-1}.$ Furthermore, $l_i=l_{i-1}$ iff each irreducible component of $Z_i$ has exactly two distinguished points, where $Z_{i}$ is the minimal elliptic subcurve of $C_{i}$.
\item[(vii)] $C_{t}$ has no disconnecting nodes.\\
\end{itemize}
These families fit into the following diagram of birational morphisms over $\Delta$
\[
\xymatrix{
&\B_{0} \ar[dr]^{q_1} \ar[dl]_{p_1}&&\B_{1}\ar[dl]_{p_1} &\!\!\!\! \cdots \cdots \cdots \!\!\!\! & \B_{t-2} \ar[dr]^{q_{t-2}} && \B_{t-1} \ar[dr]^{q_{t-1}} \ar[dl]_{p_{t-1}} &&\\
\C^{ss}:=\C_{0} \ar@{-->}[rr] && \C_{1} \ar@{-->}[r]&&\!\!\!\! \!\!\!\!   \cdots \cdots \cdots\!\!\!\! \!\!\!\! &\ar@{-->}[r]& \C_{t-1}  \ar@{-->}[rr] && \C_{t}\\
}
\]

Indeed, given $\C_{i}$ satisfying (i)-(vi), we construct $\C_{i+1}$ as follows. $C_{i}$ is Gorenstein by (i), so it possesses a minimal elliptic subcurve $Z_i \subset C_{i}$, and we define $p_{i}:\B_{i} \rightarrow \C_{i}$ to be the simple blow-up of $\C_{i}$ at the finite set of smooth points $\{p_j \,|\, p_{j} \in Z_i \}$. We define $q_{i}:\B_{i} \rightarrow \C_{i+1}$ to be the contraction of $\tilde{Z}_i$, the strict transform of $Z_i$ in $\B_{i}$. ($q_i$ is uniquely characterized by the propertes that $\Exc(q_i)=\tilde{Z}_i$ and $q_{i_*}\O_{\B_{i}}=\O_{\C_{i+1}}$.)

To prove that $q_i$ exists, consider the line-bundle
$$
\L:=\omega_{\B_i/\Delta}(\tilde{Z}_i+\sigma_1+\ldots+\sigma_n). 
$$
Note that $Z_{i} \subset \C_i$ is Cartier by (ii), so $\tilde{Z}_i \subset \B_{i}$ is Cartier. Furthermore, $\sigma_1, \ldots, \sigma_n$ are Cartier divisors on $\B_{i}$ by (iii). Adjunction and Lemma \ref{L:MinEllipticSub} give
$$\L|_{\tilde{Z}_i} \simeq \omega_{\tilde{Z}_i} \simeq \O_{\tilde{Z}_i}.$$
By (iv), $\L$ has non-negative degree on every irreducible component of $B_{i}$ not contained in $\tilde{Z}_i$, and the subcurve $E \subset C_i$ on which $\L$ has degree zero is precisely
$$E=\tilde{Z}_i \cup F,$$
where $F$ is the union of irreducible components of $B_{i}$ which are disjoint from $\tilde{Z}_i$ and have exactly two distinguished points. Now Lemma \ref{L:Contraction} applies to the line-bundle $\L$, so $\tilde{Z}_i \cup F$ is a contractible subcurve of the special fiber. Since $\tilde{Z}_i$ is disjoint from $F$, we may certainly contract $\tilde{Z}_i$ on its own; this shows that $q_i: \B_{i} \rightarrow \C_{i+1}$ exists. 

Now we must show that $\C_{i+1}$ satisfies (i)-(vii), and that after finitely many steps we achieve condition (vii). 
\begin{itemize}
\item[(i)] Locally around $q(\tilde{Z}_i)$, $\C_{i+1}$ is isomorphic to the contraction given by a high power of $\L$, so Lemma \ref{L:Contraction} implies that $C_{i+1}$ is Gorenstein. \\

\item[(ii)] Since $\C_{i}$ is regular around every node of the special fiber, so is $\B_{i}$. Since $q(\tilde{Z}_i) \in C_{i+1}$ is not a node, the same is true for $\C_{i+1}$.\\

\item[(iii)] Immediate from the fact that none of the section $\sigma_1, \ldots, \sigma_n$ on $\B_{i}$ pass through $\tilde{Z}_i$.\\

\item[(iv)] Since every component of $C_{i}$ has at least two distinguished points, and every exceptional divisor of $p_i$ has two distinguished points, every component of $B_{i}$ has at least two distinguished points. Since $q_i$ maps distinguished points to distinguished points, every component of $C_{i+1}$ has at least two distinguished points.\\

\item[(v)] Write out the fundamental decomposition of $C_{i}$:
$$C_{i}=Z_i \cup R_1 \cup \ldots \cup R_k.$$
Then we can decompose the special fiber $B_{i}$ as
$$
B_i=\tilde{Z}_i \cup \tilde{R}_1 \cup \ldots \cup \tilde{R}_k \cup F_1 \cup \ldots \cup F_j,
$$
where $\tilde{Z}_i, \tilde{R}_i$ are the strict transforms of the corresponding subcurves in $C_i$, and $F_1, \ldots, F_j$ are the exceptional curves of the blow-up. Note that $l_i=j+k.$ Lemma \ref{L:Contraction} implies that $q(\tilde{Z}_i) \in C_{i+1}$ is a Gorenstein singularity with $l_i$ branches and $\delta=l_i$. By Proposition \ref{P:genus1}, there is a unique such singularity: the elliptic $l_i$-fold point.\\

\item[(vi)] With notation as above, let $G_i \subset \tilde{R}_i$ be the unique irreducible component meeting $\tilde{Z}_i$ for each $i=1, \ldots, k$.   When $\tilde{Z}_i$ is contracted, the minimal elliptic subcurve of $C_{i+1}$ consists of the smooth rational components
$$
q(G_1) \cup \ldots \cup q(G_k) \cup q(F_1) \cup \ldots \cup q(F_j),
$$
meeting along an elliptic $l_{i}$-fold point. It is easy to see at the level $l_{i+1}$ is just the number of distinguished points of $q(G_1), \ldots, q(G_k), q(F_1), \ldots, q(F_j)$ minus $j+k$. Indeed, each component $q(G_1), \ldots, q(F_j)$ has a distinguished point where it meets the elliptic $m$-fold point and these do not contribute to $l_{i+1}$, while the remaining distinguished points are either disconnecting nodes or marked points and these each contribute one to $l_{i+1}$. Since $q$ maps distinguished points of $G_1, \ldots, G_k, F_1, \ldots, F_j$ bijectively to distinguished points of $q(G_1), \ldots, q(G_k), q(F_1), \ldots, q(F_j)$, and since each $G_1, \ldots, G_k,$ $F_1, \ldots, F_j$ has at least two distinguished points, we have $l_{i+1} \geq l_i$. Furthermore, equality holds iff each $G_1, \ldots, G_k, F_1, \ldots, F_j$ has exactly two distinguished points.\\

\item[(vii)]In the previous paragraph, we saw that if
$$C_{i}=Z_i \cup R_1 \cup \ldots \cup R_k,$$
then one irreducible component from each subcurve $R_{i}$ is absorbed into the minimal elliptic subcurve $E_{i+1} \subset C_{i+1}$. It follows that the number of irreducible components of $\overline{C_{i+1} \backslash E_{i+1}}$ is less than the number of irreducible components of $\overline{C_{i} \backslash E_{i}}.$ Thus, after finitely many steps, we have $C_{t}=E_{t}$, i.e. $C_{t}$ has no disconnecting nodes. 
\end{itemize}

\begin{Step3} Stabilize to obtain $m$-stable limit.
\end{Step3}
By (vii), $C_{t}$ has no disconnecting nodes so $l_{t}=n$. Since $m<n$, we may set $$
e:=\min\{j \,|\, l_j>m\}.$$ Let
$
\phi: \C_e \rightarrow \C
$
be the `stabilization' contraction uniquely determined by the properties that $\phi_{*}\O_{\C_e}=\O_{\C}$, and 
$$\Exc(\phi)=\{\cup_{F \subset C_e}F \,\,| \,\,\text{$F \nsubseteq Z_e$ and $F$ has exactly two distinguished points} \}.$$
Since each component $F \subset C_{e}$ satisfying the above condition is a smooth rational curve meeting the rest of the special fiber in one or two nodes, and the total space $\C_{e}$ is regular around $F$, the existence of $\phi_i$ follows by standard results on the contractibility of rational cycles \cite{Lipman}. Furthermore, the images of the sections $\sigma_1, \ldots, \sigma_n$ on $\C_{e}$ lie in the smooth locus of $\C$, so we may consider the special fiber $(C,p_1, \ldots, p_n)$ as an $n$-pointed curve. To show that $(C, p_1, \ldots, p_n)$ is $m$-stable, we must verify conditions (1)-(3) of Definition \ref{D:mstable}.
\begin{itemize}
\item[(1)] \emph{$C$ has only nodes and elliptic-$l$ fold points, $l \leq m$, as singularities.} By conditions (i) and (v) above, $C_{e}$ has only nodes and an elliptic $l_{e-1}$-fold point as singularities, where $l_{e-1}<m$ by our choice of $e$. The same is true of $C$, since the only singularities produced by contracting semistable chains of rational curves are nodes.\\
\item[(2)]  \emph{$C$ has level $>m$.} The level of $C_{e}$ is $>m$ by our choice of $e$, so it suffices to see that the level of $C$ is the same as the level of $C_e$. Let
$$
C_e=Z_{e} \cup R_1 \cup \ldots \cup R_k,
$$
be the fundamental decomposition of $C_{e}$. Order the $R_{i}$ so that $R_1, \ldots, R_j$ consist entirely of components with two distinguished points, while $R_{j+1}, \ldots, R_{k}$ each contain a component with $\geq 3$ distinguished points. Then $\phi$ contracts each of $R_1, \ldots, R_j$ to a point, so that the fundamental decomposition of $C$ is
$$
C=\phi(Z_{e}) \cup \phi(R_{j+1}) \cup \ldots \cup \phi(R_k).
$$
Thus,
\begin{align*}
&|\overline{C \backslash \phi(Z_{e})}|=|\overline{C_e \backslash Z_{e}}|-j.
\end{align*}
On the other hand, since each $R_1, \ldots, R_j$ must be a chain of $\P^{1}$'s whose final component carries a marked point, $\phi(R_1), \ldots, \phi(R_j)$ will be marked points on the minimal elliptic subcurve $\phi(Z_{e})$, i.e. we have
$$
|\{p_i \,| \,p_i \in \phi(Z_{e})\}|= |\{p_i \,| \,p_i \in Z_{e}\}|+j.
$$
Thus, $|\overline{C \backslash \phi(Z_{e})}|+|\{p_i \,| \,p_i \in \phi(Z_{e})\}|=|\overline{C_e \backslash Z_{e}}|+ |\{p_i \,| \,p_i \in Z_{e}\}|$ as desired.\\
\item[(3)]  \emph{$(C,p_1, \ldots, p_n)$ satisfies the stability condition.} Since $\phi$ contracts every component of $R_1 \cup \ldots \cup R_k$ with two distinguished points, every component of $ \phi(R_1) \cup \ldots \cup \phi(R_k)$ has at least three distinguished points. It remains to check the stability condition for irreducible components of $\phi(Z_{e})$.

We may assume that $e \geq 1$, so $Z_{e}$ consists of $l_{e-1}$ smooth rational branches meeting in an elliptic $l_{e-1}$-fold point. Since no component of $Z_{e}$ is contained in $\Exc(\phi)$, $Z_{e}$ maps isomorphically onto $\phi(Z_{e})$ and condition (iv) implies that every component of $\phi(Z_{e})$ has at least two distinguished points. Finally, if every component of $\phi(Z_{e})$ had exactly two distinguished points, the same would be true of $Z_{e}$ and condition (vi) would imply that $l_i=l_{i-1}$. This contradicts our choice of $e$; we conclude that some component of $\phi(Z_{e})$ has at least three distinguished points.
\end{itemize}

\subsubsection{Uniqueness of $m$-stable Limits}
In order to prove that an isomorphism
$$
(\C,\sigma_1, \ldots, \sigma_n)|_{\eta} \simeq (\C',\sigma_1', \ldots, \sigma_n')|_{\eta}
$$
extends to an isomorphism over $\Delta$, it suffices to check that the rational map $\C \dashrightarrow \C'$ extends to an isomorphism after a finite base-change. Thus, we may assume that there exists a flat proper nodal curve $(\C^{ss} \rightarrow \Delta, \tau_1, \ldots, \tau_n)$ with regular total space and a diagram
\[
\xymatrix{
&(\C^{ss}, \tau_1, \ldots, \tau_n) \ar[rd]^{\phi'} \ar[ld]_{\phi} &\\
(\C,\sigma_1, \ldots, \sigma_n)& &(\C', \sigma'_1, \ldots, \sigma'_n)
}
\]
where $\phi$ and $\phi'$ are proper birational morphisms over $\Delta$. In fact, we may further assume that $(\C^{ss} \rightarrow \Delta, \tau_1, \ldots, \tau_n)$ is Deligne-Mumford semistable. Indeed, any unmarked (-1)-curve in the special fiber $C^{ss}$ must be contracted by both $\phi$ and $\phi'$ since neither $C$ nor $C'$ contain unmarked smooth rational components meeting the rest of the curve in a single point. Thus, $\phi$ and $\phi'$ both factor through the minimal model of $\C^{ss}$, obtained by successively blowing down unmarked (-1)-curves.

The strategy of the proof is to show that $\Exc(\phi)=\Exc(\phi')$. Since $\C$ and $\C'$ are normal, this immediately implies $\C \simeq \C'$ . The proof proceeds in three steps: In step 1, we handle the case where either $C$ or $C'$ is a nodal curve. After step 1, we may assume that $C$ and $C'$ each have a non-nodal singular point, say $p$ and $p'$, and we set
\begin{align*}
E&:=\phi^{-1}(p) \subset C^{ss}\\
E'&:=\phi'^{-1}(p') \subset C^{ss}.
\end{align*}
Using the classification of semistable tails of the elliptic $m$-fold point (Proposition \ref{P:semistablelimits}), we show that $E=E'$. Finally, in step 3, we show that $E=E'$ implies $\Exc(\phi)=\Exc(\phi')$.

\begin{Step1} The case when $C$ or $C'$ contains is nodal.
\end{Step1}
We may assume that $C'$ is nodal, but that $C$ contains an elliptic $l$-fold point $p$ for some $l \leq m$. Indeed, if $C$ and $C'$ are both nodal, then they are Deligne-Mumford stable, so $\C \simeq \C'$ by usual stable reduction theorem. Now set
$$
E:=\phi^{-1}(p) \subset C^{ss},
$$
and note that $p_a(E)=1$ and $|E \cap \overline{C^{ss} \backslash E}|=l \leq m.$ It follows that $\phi'(E) \subset C'$ is an unmarked connected arithmetic genus one subcurve meeting $\overline{C' \backslash \phi'(E')}$ in no more than $m$ points, which contradicts the $m$-stability of $C'$.

\begin{Step2} $E=E'$.
\end{Step2}
By step 1, we may assume that $C$ and $C'$ each have a non-nodal singular point, say $p$ and $p'$, and we set
\begin{align*}
E&:=\phi^{-1}(p) \subset C^{ss},\\
E'&:=\phi^{-1}(p') \subset C^{ss}.
\end{align*}
We invoke Proposition \ref{P:semistablelimits}, which says that $(E,q_1, \ldots, q_k)$ and $(E',q'_1, \ldots, q'_l)$ are \emph{balanced}, where
\begin{align*}
\{q_1, \ldots, q_k\}:&=\{E \cap \overline{C^{ss} \backslash E}\}\\
\{q'_1, \ldots, q'_l\}:&=\{E' \cap \overline{C^{ss} \backslash E'}\}
\end{align*}
Let $Z \subset C^{ss}$ be the minimal elliptic subcurve of $C^{ss}$. By Corollary \ref{C:minimality}, we have $Z \subset E$ and $Z \subset E'$. Proposition \ref{P:semistablelimits} implies there exist integers $l$ and $l'$ such that
\begin{align*}
l&:=l(Z,q_1)=\ldots=l(Z,q_k)\\
l'&:=l(Z,q'_1)=\ldots=l(Z,q'_l)
\end{align*}
Put differently, this says that $E$ comprises all components in $C^{ss}$ whose length from $Z$ is less than $l$, while $E'$ comprises all irreducible components in $C^{ss}$ whose length from $Z$ is less than $l'$. If $l=l'$, then we have $E=E'$ and we are done. Otherwise, we may assume that $l<l'$, and we have a strict containment $E \subset E'$. But then, since $E'$ meets $\overline{C^{ss} \backslash E'}$ in no more than $m$ points, $\phi(E') \subset C$ is a connected arithmetic genus one subcurve meeting $\overline{C \backslash \phi(E')}$ in no more than $m$ points. This contradicts the $m$-stability of $C$.
\begin{Step3}
$\Exc(\phi)=\Exc(\phi')$
\end{Step3}
It is enough to show that $E$ and $E'$ determine $\Exc(\phi)$ and $\Exc(\phi')$ in the following sense:
\begin{align*}
\Exc(\phi):&=E \, \cup \{F \, |  \,F \cap E = \emptyset \text{ and $F$ has two distinguished points }\}\\
\Exc(\phi'):&=E' \cup \{F \,| \, F \cap E' = \emptyset \text{ and $F$ has two distinguished points }\}
\end{align*}
Let us argue the first equality (the argument for the second is identical).

It is clear that no irreducible component of $C^{ss}$ which meets $E$ can be contracted by $\phi$. Such a component would be contracted to the point $p$ and hence contained in $E:=\phi^{-1}(p).$ It remains to see that an irreducible component $F \subset C^{ss}$ with $F \cap E = \emptyset$ is contracted iff $F$ has exactly two distinguished points. If $F$ has at least three distinguished points, then it cannot be contracted without introducing: a singular point with more than three branches, a section passing through a node, or two sections colliding, any one of which contradicts the $m$-stability of $C$. On the other hand, if $F$ has two distinguished points, then $F$ must be contracted or else $\phi(F) \subset C$ is an irreducible component lying outside the minimal elliptic subcurve and containing only two distinguished points. This completes the proof.
\subsection{Valuative criterion for $\SM_{1,\A}(m)$}\label{S:mAstability}
In this section, we complete the proof of Theorem \ref{T:Valuative} by handling the case when $\A \neq (1, \ldots, 1)$. The key idea, following Hassett \cite{Hassett4},  is that we can construct the $(m,\A)$-stable limit from the $m$-stable limit by running a relative minimal model program with respect to $\omega_{\C/\Delta}(\Sigma_{i}a_i\sigma_i)$.

\subsubsection{Existence of Limits}
Given a family of smooth $n$-pointed curves over the generic point of the spectrum of a discrete valuation ring $\Delta$, we may (after a finite base-change) complete this family to an $m$-stable curve $(\pi:\C \rightarrow \Delta, \sigma_1, \ldots, \sigma_n)$. To obtain the $(m,\A)$-stable limit, we construct a sequence of birational contractions
$$\C:=\C_{0} \rightarrow \C_{1} \rightarrow \ldots \rightarrow \C_{N},$$
where each special fiber $C_{i}$ satisfies conditions (1)-(4) of Definition \ref{D:mstable}, and such that $\omega_{\C_N/\Delta}$ is relatively ample. Thus, $\C_{N} \rightarrow \Delta$ is the desired $(m,\A)$-stable limit. 

To construct this sequence of contractions, we proceed by induction on $i$. If $\omega_{\C_i/\Delta}(\Sigma_{i}a_i\sigma_i)$ is ample, we are done. If not, then $\omega_{\C_i/\Delta}(\Sigma_{i}a_i\sigma_i)$ has non-positive degree on some component of the special fiber, and we claim that this component must be a smooth rational curve meeting the rest of the fiber in a single node. To see this, note that condition (3) of Definition \ref{D:mstable} implies that every component $F \subset C_{i}$ satisfies one of the following:
\begin{itemize}
\item[(I)] $F$ has arithmetic genus one and at least one distinguished point.
\item[(II)] $F$ is a smooth rational component meeting an elliptic $l$-fold point and has at least one additional distinguished point.
\item[(III)] $F$ is a smooth rational component meeting the rest of the fiber in at least two nodes and has at least one additional distinguished point.
\item[(IV)]  $F$ is a smooth rational component meeting the rest of the fiber in one node.
\end{itemize}
On components of type (I)-(III), the restriction of the dualizing sheaf $\omega_{\C/\Delta}|_{F}$ has non-negative degree. Since the weights $a_{i}$ are each positive, each distinguished point contributes a positive amount to the degree, and we conclude that $\omega_{\C_i/\Delta}(\Sigma_{i}a_i\sigma_i)$ has positive degree on all such components. Thus, if $\omega_{\C_i/\Delta}(\Sigma_{i}a_i\sigma_i)$ fails to be ample, it has non-positive degree on a component of type (IV). If $F \subset C_{i}$ is such a component, standard results on the contractibility of rational cycles imply the existence of a projective birational contraction $\phi: \C_{i} \rightarrow \C_{i+1}$ contracting $F$ to a smooth point  \cite{Lipman}. 

Let us check that $\C_{i+1}$ still satisfies conditions (1)-(4) of Definition \ref{D:mstable}. Condition (1) is clear since $\phi(F) \in C_{i+1}$ is a smooth point. For condition (2), we claim that the level of $C_{i+1}$ is the same as the level of $C_{i}$. To see this, note that $F$ does not belong to the minimal elliptic subcurve $Z_{i} \subset C_{i}$, and consider two cases. If $\phi(F)$ is not contained in the minimal elliptic subcurve $Z_{i+1} \subset C_{i+1}$, then clearly the level is unchanged. On the other hand, if $\phi(F) \in Z_{i+1}$, then the fact that $F$ must contain at least one marked point implies
\begin{align*}
|Z_{i+1} \cap \overline{C_{i+1} \backslash Z_{i+1}}|&=|Z_i \cap \overline{C_i \backslash Z_i}|-1\\
|Z_{i+1} \cap \Sigma_{i+1}|&=| Z_i \cap \Sigma_i |+1,
\end{align*}
where $\Sigma_i$ is the support of the divisor of marked points on $C_{i}$. Thus, the level $|Z_i \cap \overline{C_i \backslash Z_i}|+|\Sigma_{i} \cap Z_i|$  is again unchanged. For condition (3), simply note that every component in $C_{i+1}$ has as many distinguished points as its strict transform in $C_{i}$. Finally, for condition (4), we must check that if $p_{i_1}, \ldots, p_{i_k}$ are the marked points supported on $E$, then $\sum_{j=1}^{k}a_{i_j} \leq 1$. This is clear since we chose a component $F$ on which $\omega_{C_i}(\Sigma_{i}a_ip_i)$ had non-positive degree.

Since there are only finitely many components in the special fiber of $\C$, and the total degree of $\omega_{\C/\Delta}(\Sigma_ia_i\sigma_i)$ is positive, we must achieve ampleness of $\omega_{\C_{i}/\Delta}(\Sigma_ia_i\sigma_i)$ after finitely many repetitions of this procedure.

\subsubsection{Uniqueness of Limits}
To prove uniqueness of $(m,\A)$-stable limits, it suffices (by uniqueness of $m$-stable limits) to show the following: Given an $(m,\A)$-stable curve $(\pi':\C' \rightarrow \Delta, \sigma_1, \ldots, \sigma_n)$ with smooth generic fiber, there exists an $m$-stable curve $(\pi:\C \rightarrow \Delta, \sigma_1, \ldots, \sigma_n)$ and a birational morphism $\C \rightarrow \C'$ such that
$$\C'=\Proj \oplus_{m \geq 0} \pi_*\left(\omega_{\C/\Delta}(\Sigma_{i}a_i\sigma_i)^m\right),
$$
where the sum is taken over $m$ sufficiently divisible so that $\omega_{\C/\Delta}(\Sigma_{i}a_i\sigma_i)^m$ is integral.

To obtain $\C \rightarrow \C',$ simply apply stable reduction locally around the points of $C'$ where marked points coincide. This gives a diagram of birational morphisms \[
\xymatrix{
&\C^{ss} \ar[dr]^{\phi_1} \ar[dl]_{\phi_2}&\\
\C \ar[rr]^{\phi}&&\C'
}
\]
satisfying:
\begin{itemize}
\item $\phi_1$ is a composition of blow-ups along smooth points of the special fiber.
\item $\phi_{1}(\Exc(\phi_1)) \in C'$ is the locus where two or more marked points coincide.
\item  $\phi_2$ is the contraction of all unmarked $(-2)$-curves in $\Exc(\phi_1)$.
\item The strict transforms of $\sigma_1, \ldots, \sigma_n$ on $\C$ are disjoint.
\item $\omega_{\C/\Delta}(\Sigma_i\sigma_i)$ is $\phi$-ample.
\end{itemize}
We claim that $(\C \rightarrow \Delta, \sigma_1, \ldots, \sigma_n)$ is an $m$-stable curve. By construction, the sections $\sigma_1, \ldots, \sigma_n$ are distinct, and $\omega_{\C/\Delta}(\Sigma_i\sigma_i)$ is relatively ample, so it suffices to check conditions (1)-(3) of Definition \ref{D:mstable}. For condition (1), since $C'$ has only nodes and elliptic $l$-fold points, the same is true of $C$. For condition (2), we will show that the level of $C$ is the same as the level of $C'$. To see this, let $Z \subset C$ be the minimal elliptic subcurve of $C$, we write the fundamental decomposition
$$
C=Z \cup R_1 \cup \ldots \cup R_k.
$$
We may order the $R_{i}$ so that $R_{1}, \ldots, R_{j}$ are contracted to a point by $\phi$, while $R_{j+1}, \ldots, R_k$ are not. Then
$$
C'=\phi(Z) \cup \phi(R_{j+1}) \cup \ldots \cup \phi(R_{k})
$$
is the fundamental decomposition of $C'$, so we have 
$$|\phi(Z) \cap \overline{C' \backslash \phi(Z)}|=|Z \cap \overline{C \backslash Z}|-j.$$
 On the other hand, since each rational chain $R_{i}$ must support at least one marked point, the points $\phi(R_{1}), \ldots, \phi(R_{j})$  are now marked distinguished points on $\phi(Z)$. Thus, 
 $$|\phi(Z) \cap \Sigma' |=|Z \cap \Sigma|+j.$$
 In sum, we get, $$|\phi(Z) \cap \overline{C' \backslash \phi(Z)}|+|\Sigma' \cap \phi(Z)|=|Z \cap \overline{C \backslash Z}|+|\Sigma \cap Z|,$$ as desired. Finally, condition (3) is immediate from the fact that each irreducible component of $\Exc(\phi)$ has at least three distinguished points.

To see that $\C'=\Proj \oplus_{m \geq 0} \pi_*\left(\omega_{\C/\Delta}(\Sigma_{i}a_i\sigma_i)^m\right),
$
we only need to check that 
\begin{align*}
\omega_{\C^{ss}/\Delta}(\Sigma_{i}a_i\sigma_i)-\phi_1^*\,\omega_{\C'/\Delta}(\Sigma_i a_i\sigma_i) \geq 0,\\
\omega_{\C^{ss}/\Delta}(\Sigma_{i}a_i\sigma_i)-\phi_2^*\,\omega_{\C/\Delta}(\Sigma_i a_i\sigma_i) \geq 0.
\end{align*}
Indeed, this implies that 
$$
\pi'_*\left(\omega_{\C'/\Delta}(\Sigma_{i}a_i\sigma_i)^m \right)=\pi^{ss}_*\left(\omega_{\C^{ss}/\Delta}(\Sigma_{i}a_i\sigma_i)^m\right)=\pi_*\left(\omega_{\C/\Delta}(\Sigma_{i}a_i\sigma_i)^m\right)
$$
for all $m>>0$ sufficiently divisible. Since $\omega_{\C'}(\sum_{i}a_i\sigma_i)$ is an ample $\Q$-divisor, this gives
\begin{align*}
\C'&=\Proj \oplus_{m \geq 0} \pi'_*\left(\omega_{\C'/\Delta}(\Sigma_{i}a_i\sigma_i)^m\right)\\&= \Proj \oplus_{m \geq 0} \pi^{ss}_*\left(\omega_{\C^{ss}/\Delta}(\Sigma_{i}a_i\sigma_i)^m\right)\\&=\Proj \oplus_{m \geq 0} \pi_*\left(\omega_{\C/\Delta}(\Sigma_{i}a_i\sigma_i)^m\right).
\end{align*}
Since $p_{i_1}, \ldots, p_{i_k} \in C'$ coincide only if $\sum_{j=1}^{k}a_{i_j} \leq 1$, $\phi_1$ is composed of blow-ups at smooth points where the total multiplicity of $\sum_{i}a_i\sigma_i$ is less than or equal to one, which gives
$$
\omega_{\C^{ss}/\Delta}(\Sigma_{i}a_i\sigma_i)-(\phi_1)^*\omega_{\C'/\Delta}(\Sigma_i a_i\sigma_i) \geq 0.
$$
On the other hand, since $\phi_2$ is simply a contraction of unmarked $(-2)$-curves, we have
$$
\omega_{\C^{ss}/\Delta}(\Sigma_{i}a_i\sigma_i)-(\phi_2)^*\omega_{\C'/\Delta}(\Sigma_i a_i\sigma_i) = 0.
$$

\appendix
\section{Gorenstein curve singularities of genus one}
Let $C$ be a curve over an algebraically closed field $k$, $p \in C$ a singular point, and $\pi: \tilde{C} \rightarrow C$ be the normalization of $C$ at $p$. We have the following basic numerical invariants.
\begin{definition}
\begin{align*}
\delta(p)&:=\dim_k \pi_*\O_{\tilde{C},p}/\O_{C,p}\\
m(p)&:=|\pi^{-1}(p)|\\
g(p)&:=\delta(p)-m(p)+1\\
\end{align*}
\end{definition}
We call $g(p)$ the \emph{genus} of the singularity. Note that if $C$ is complete and has arithmetic genus $g$, then $g(p) \leq g$. The purpose of this appendix is to classify (up to analytic isomorphism) Gorenstein singularities of genus zero and one. The main results are

\begin{proposition}\label{P:genus0}
If $p \in C$ has $m$ branches and genus zero, then
$$\hat{O}_{C,p} \simeq k[[x_1, \ldots, x_m]]/I,$$ where
$$
I:=(x_ix_j: 1 \leq i<j \leq m).
$$
Furthermore, $p$ is Gorenstein iff $m=2$. (i.e. when $p$ is an ordinary node.)
\end{proposition}
\begin{proposition}\label{P:genus1}
If $p \in C$ is Gorenstein with $m$ branches and genus one, then $p$ is an elliptic $m$-fold point, i.e.
$$
\hat{O}_{C,p} \simeq
\begin{cases}
k[[x,y]]/(y^2-x^3) & m=1\\
k[[x,y]]/y(y-x^2) & m=2 \\
k[[x,y]]/xy(y-x) & m=3\\
k[[x_1, \ldots, x_{m-1}]]/I_m & m \geq 4,
\end{cases}
$$
where $I_{m}$ is the ideal generated by all quadrics of the form
\begin{align*}
x_{h}(x_i-x_j) \text{ with } i,j,h \in \{1, \ldots, m-1\} \text{ distinct.}
\end{align*}
\end{proposition}
\begin{remark}
There are many non-isomorphic non-Gorenstein singularities of genus one with fixed number branches. Furthermore, in higher genus, there are many non-isomorphic Gorenstein singularities with fixed number branches.
\end{remark}
Combining these two propositions, we conclude
\begin{corollary}\label{C:gorenstein}
If $C$ is a Gorenstein curve with $p_a(C)=1$, and $p \in C$ is a singular point, then $p$ is either an ordinary node or an elliptic $m$-fold point for some integer $m$.
\end{corollary}

In order to prove the propositions, it will be useful to switch to ring-theoretic notation. Set
\begin{align*}
&R:=\hat{\O}_{C,p}.\\
&\tilde{R}:=\widetilde{R/P_1} \oplus \ldots \oplus \widetilde{R/P_{k(p)}},
\end{align*}
where $P_1, \ldots, P_{m}$ are the minimal primes of $R$, and $\widetilde{R/P_i}$ denotes the integral closure of $R/P_i$. Note that
$$
\tilde{R} \simeq k[[t_1]] \oplus \ldots \oplus k[[t_m]],
$$
since each $\widetilde{R/P_i}$ is a complete, regular local ring of dimension one over $k$. Let $m_R$ be the maximal ideal of $R$, and let $m_{\tilde{R}}$ be the ideal $(t_1) \oplus \ldots \oplus (t_m)$. Since $R$ is reduced, we have an embedding
\begin{align*}
R &\hookrightarrow \tilde{R},\\
m_R&=(m_{\tilde{R}} \cap R).
\end{align*}
In these terms, the \emph{conductor ideal} of the singularity is given by
$$
I_{p}:=\Ann_{R}(\tilde{R}/R),
$$
and $R$ is \emph{Gorenstein} iff (\cite{Serre})
$$
\dim_k (R/I_{p})=\dim_k(\tilde{R}/R).
$$

Note that the $R$-module $\tilde{R}/R$ has a natural grading given by powers of $m_{\tilde{R}}$; we define
$$
(\tilde{R}/R)^{i}=m_{\tilde{R}}^i/((m_{\tilde{R}}^i \cap R)+m_{\tilde{R}}^{i+1}),
$$
Now we have the following trivial observations:
\begin{itemize}
\item[(1)] $\delta(p)=\sum_{i \geq 0} \dim_{k} (\tilde{R}/R)^i$
\item[(2)] $g(p)=\sum_{i \geq 1} \dim_{k} (\tilde{R}/R)^i$
\item[(3)] $(\tilde{R}/R)^i=(\tilde{R}/R)^j=0 \implies (\tilde{R}/R)^{i+j}=0$ for any $i,j \geq 1$.
\end{itemize}
Having dispensed with these preliminaries, the proofs of propositions \ref{P:genus0} and \ref{P:genus1} are straightforward, albeit somewhat tedious. The basic idea is to find a basis for $m_{R}/m_{R}^2$ in terms of the local coordinates $t_1, \ldots, t_m$.

\begin{proof}[Proof of Proposition \ref{P:genus0}]
If $g(p)=0$, then $(\tilde{R}/R)^i=0$ for all $i>0$, so $m_{R}=m_{\tilde{R}}$. Thus, we may define a local homomorphism of complete local rings
\begin{align*}
k[[x_1, \ldots, x_m]]&\rightarrow R \subset k[[t_1]] \oplus \ldots \oplus k[[t_m]]\\
x_i &\rightarrow (0, \ldots,0, t_i, 0, \ldots 0)\\
\end{align*}
This homomorphism is surjective since it is surjective on tangent spaces, and the kernel is precisely the ideal
$$
I_{m}=(x_ix_j, i<j).
$$
To see that $R$ is Gorenstein iff $m=2$, note that the conductor ideal is
$$I_p=m_{R}.$$
Thus, the Gorenstein condition
$$
\dim_k(\tilde{R}/R)=\dim_{k}(R/I_p)
$$
is satisfied iff $\dim_k(\tilde{R}/R)=1$, i.e. when $m=2$.
\end{proof}

\begin{proof}[Proof of Proposition \ref{P:genus1}]
Since $g(p)=1$, observations (2) and (3) imply that
\begin{align*}
&\dim_{k} (\tilde{R}/R)^1=1\\
&\dim_{k} (\tilde{R}/R)^i=0 \text{ for all $i \geq 1$}.
\end{align*}
Put differently, this says that
$$
m_{R} \supset m_{\tilde{R}}^2,
$$
while
$$
m_{R}/m_{\tilde{R}}^2 \subset m_{\tilde{R}}/m_{\tilde{R}}^2
$$
is a codimension-one subspace. By Gaussian elimination, we may choose elements $f_1, \ldots, f_{m-1} \in m_{R}$ such that
\begin{equation*}
\left(
\begin{matrix}
f_1\\
\vdots\\
\vdots\\
f_{m-1}
\end{matrix}
\right)\equiv
\left(
\begin{matrix}
t_1& 0& \hdots  & 0 & a_1t_{m-1}\\
0&t_2& \ddots & \vdots & a_2t_{m-1} \\
\vdots& \ddots& \ddots& 0& \vdots  \\
0 & \hdots &0 & t_{m-2} & a_{m-1}t_{m-1}
\end{matrix}
\right)
\begin{matrix}
\mod m_{\tilde{R}}^2
\end{matrix}
\end{equation*}
for some $a_1, \ldots, a_{m-1} \in k.$ 
\begin{claim}
If $R$ is Gorenstein, we may take $a_1, \ldots, a_{m-1}=1$.
\end{claim}
\begin{proof}[Proof of Claim]
First, let us show that $R$ Gorenstein implies $I_{p}=m_{\tilde{R}}^{2}$. Since $m_{R} \supset m_{\tilde{R}}^2$, we certainly have $I_p \supset m_{\tilde{R}}^2$. Thus,
$$
\dim (R/I_p) \leq \dim (R/m_{\tilde{R}}^2)=m.
$$
On the other hand, we have $\dim (\tilde{R}/R)=m$, so the Gorenstein equality  $\dim (R/I_p)=\dim (\tilde{R}/R)$ implies $\dim (R/I_p) = \dim R/m_{\tilde{R}}^2$, i.e. $I_p=m_{\tilde{R}}^2$.

In particular, we have $f_1, \ldots, f_{m-1} \notin I_p$. Now if $a_i=0$ then
$$
f_ig \in (f_i) + m_{\tilde{R}}^2 \subset R, \text{ for all $g \in \tilde{R}$,}
$$
i.e. $f \in I_p$. We conclude that $a_i \in k^{*}$ for each $i=1, \ldots, m$. Making a change of coordinates $t_i'=a_it_i$, we may assume that each $a_i=1$.
\end{proof}
At this point, the proof breaks into three cases:
\begin{itemize}
\item[I.] $(m \geq 3)$ We claim that $f_1, \ldots, f_{m-1}$ give a basis for $m_{R}/m_{R}^2$. Clearly, it is enough to show that $m_{R}^2=m_{\tilde{R}}^2$. Since $m_{R}^2 \supset m_{\tilde{R}}^4$, it is enough to show that
$$m_{R}^2/m_{\tilde{R}}^4 \hookrightarrow m_{\tilde{R}}^2/m_{\tilde{R}}^4$$ is surjective.
Using the matrix expressions for the $\{f_i\}$, one easily verifies that $f_1^{2}, \ldots, f_{m-1}^2, f_1f_2$ map to a basis of $m_{\tilde{R}}^2/m_{\tilde{R}}^3$, and $f_1^{3}, \ldots, f_{m-1}^3, f_1^2f_2$ map to a basis of $m_{\tilde{R}}^3/m_{\tilde{R}}^4$.

Since $f_1, \ldots, f_{m-1}$ give a basis of $m_{R}/m_{R}^2$, we have a surjective hoomomorphism
\begin{align*}
k[[x_1, \ldots, x_{m-1}]] &\rightarrow R \subset k[[t_1]] \oplus \ldots \oplus k[[t_m]]\\
x_i &\rightarrow (0, \ldots,0, t_i, 0, \ldots 0,t_{m-1}),
\end{align*}
and the kernel is precisely
$
I=(x_{h}(x_i-x_j) \text{ with } i,j,h \in \{1, \ldots, m-1\} \text{ distinct}).
$\\

\item[II.] $(m=2)$ By the preceeding analysis, there exists $f_1 \in m_{R}$ such that 
$$f_1 \equiv (t_1\,\,\, t_2) \mod m_{R}^2.$$
Since $m_{R} \supset m_{\tilde{R}}^2$, we may choose $f_2 \in m_{R}$ such that $f_1^2, f_2$ map to a basis of $m_{\tilde{R}}^2/m_{\tilde{R}}^3$. After Gaussian elimination, we may assume that
\begin{equation*}
\left(
\begin{matrix}
f_1^2\\

f_{2}
\end{matrix}
\right)\equiv
\left(
\begin{matrix}
t_1^2& t_2^2\\
0 & t_2^2\\
\end{matrix}
\right)
\begin{matrix}
\mod m_{\tilde{R}}^3
\end{matrix}
\end{equation*}
We claim that $f_1$ and $f_2$ form a basis for $m_{R}/m_{R}^2$. Since $f_1, f_2, f_1^2$ form a basis for $m_{R}/m_{\tilde{R}}^3$, it suffices to show that $m_{R}^2 \cap m_{\tilde{R}}^3=m_{\tilde{R}}^3$. Since $m_{R}^2 \supset m_{\tilde{R}}^4$, it is enough to show that
$$
(m_{R}^2 \cap m_{\tilde{R}}^3 )/m_{\tilde{R}}^4 \hookrightarrow m_{\tilde{R}}^3/m_{\tilde{R}}^4
$$
is surjective. From the matrix expression for the $\{f_i\}$, one easily sees that $f_1^3$ and $f_1f_2$ give a basis of $m_{\tilde{R}}^3/m_{\tilde{R}}^4$.

Since $f_1, f_2$ give a basis of $m_{R}/m_{R}^2$, we have a surjective homomorphism of complete local rings
\begin{align*}
k[[x,y]] &\rightarrow R \subset k[[t_1]] \oplus k[[t_2]]\\
x &\rightarrow (t_1,t_2),\\
y &\rightarrow (0,t_2^2),
\end{align*}
with kernel $y(y-x^2)$.\\

\item[III.] $(m=1)$ Since $m_{R}/m_{\tilde{R}}^2 \subset m_{\tilde{R}}/m_{\tilde{R}}^2$ is codimension-one, we have $m_{R}=m_{\tilde{R}}^2$. Thus, we may pick $f_1, f_2 \in m_{R}$ so that \begin{equation*}
\left(
\begin{matrix}
f_1\\
f_{2}
\end{matrix}
\right)\equiv
\left(
\begin{matrix}
t_1^2\\
t_1^3\\
\end{matrix}
\right)
\begin{matrix}
\mod m_{\tilde{R}}^4.
\end{matrix}
\end{equation*}
Since $m_{R}^2 = m_{\tilde{R}}^4$, $f_1$ and $f_2$ give a basis for $m_{R}/m_{R}^2$. Thus, the homomorphism
\begin{align*}
k[[x,y]]&\rightarrow R \subset k[[t_1]] \\
x &\rightarrow (t_1^2),\\
y &\rightarrow (t_1^3),
\end{align*}
is surjective, with kernel $y^2-x^3$.

\end{itemize}
\end{proof}


\begin{thebibliography}{}
\bibitem{DM} P. Deligne, D. Mumford, \emph{The irreducibility of the space of curves of given genus}, Inst. Hates \'{E}tudes Sci. Publ. Math. (1969) no. 36, 75-109.
\bibitem{EGAIV} J. Dieudonn\'{e} and A. Grothendieck, \emph{\'{E}l\'{e}ments de g\'{e}om\'{e}trie alg\'{e}brique IV}, Pub. Math. Inst. Hautes \'{E}tudes Sci., 1965-.
\bibitem{Hassett4}
B. Hassett, \emph{Moduli spaces of weighted pointed curves}, Advances in Mathematics 173 (2003), no. 2, 316-352.
\bibitem{H2} B. Hassett, \emph{Classical and minimal models of the moduli space of cuves of genus two}, Geometric methods in algebra and number theory, Progr. Math., 235, Birkhauser Boston (2005), 169-192.
\bibitem{H05}
B. Hassett, D. Hyeon, \emph{Log canonical models of the moduli space of curves: the first contraction}, to appear in Transactions of the AMS.
\bibitem{H:tacnodes}
B. Hassett, D. Hyeon, \emph{Log canonical models of the moduli space of curves: the first flip}, in preparation.
\bibitem{HL}
D. Hyeon, Y. Lee, \emph{Log minimal model program for the moduli space of stable curves of genus three}, submitted.
\bibitem{Kleiman} S. Kleiman. \emph{Relative duality for quasicoherent sheaves}, Compositio Math. 41 (1980) no. 1, 39-60.
\bibitem{Kollar2} J. Kollar, \emph{Rational Curves on Algebraic Varieties}, Ergebnisse der Mathematik und ihrer Grenzgebiete. 3. Folge, 32. Springer-Verlag, Berlin, 1996.
\bibitem{Lipman} J. Lipman, \emph{Rational Singularities}, Publ. Math. I.H.E.S., no. 36 (1969).
\bibitem{LMB} G. Laumon and L. Moret-Bailly, \emph{Champs alg\'{e}braique}, Ergebnisse der Mathematik und ihrer Grenzgebiete, 39. Springer-Verlag, Berlin (2000).
\bibitem{GIT} D. Mumford, J. Fogary, F. Kirwan, \emph{Geometric Invariant Theory}, Ergebnisse der Mathematik und ihrer Grenzgebiete, 34. Springer-Verlag, Berlin (1991).
\bibitem{Pinkham} H. Pinkham. \emph{Deformations of algebraic varieties with $\mathbb{G}_{m}$-action}, Ast\'{e}risque, No. 20. Soci\'{e}t\'{e} Math\'{e}matique de France, Paris, 1974.
\bibitem{Sch} M. Schlessinger, \emph{Infintesimal deformations of singularities}, Thesis, Harvard University, 1964.
\bibitem{S91}
D. Schubert, \emph{A new compactification of the moduli space of curves}, Compositio Math. 78 (1991), no. 3, 297-313.
\bibitem{Serre} J.P. Serre, \emph{Algebraic Groups and Class Fields}, Springer-Verlag, 1988.
\bibitem{me2}
D. Smyth, \emph{Modular compactifications of $M_{1,n}$ II: The log minimal model program for $\overline{M}_{1,n}$}, in preparation.
\end{thebibliography}
\end{document}